\documentclass[com]{article}
\usepackage{tikz}
\usepackage{amsmath, amssymb, amsthm}
\usepackage{graphicx}
\usepackage[all]{xy}
\newtheorem{definition}{Definition}[section]
  \newtheorem{theorem}[definition]{Theorem}
  \newtheorem{lemma}[definition]{Lemma}
  \newtheorem{corollary}[definition]{Corollary}
  \newtheorem{proposition}[definition]{Proposition}
  \newtheorem{remark}[definition]{Remark}
  \newtheorem{example}[definition]{Example}
    \newtheorem{open}[definition]{Open Problem}
    \newtheorem{property}[definition]{Property}

\title{On posetal and complete\\ partial applicative structures}
\author{Samuele Maschio}
\date{}
\begin{document}
\maketitle
\begin{abstract}
Every partial applicative structure gives rise to an indexed binary relation, that is a contravariant functor from the category of sets to the category of sets endowed with binary relations and maps preserving them. In this paper we characterize those partial applicative structures giving rise to indexed relations satisfying certain elementary properties in terms of algebraic or computational properties. We will then provide a characterization of those partial applicative structures giving rise to indexed preorders and indexed posets, and we will relate the latter ones to some particular classes of unary partial endofunctions. We will analyze the relation between a series of computational and algebraic properties in the posetal case. Finally, we will study the problem of existence of suprema in the case of partial applicative structures giving rise to indexed preorders, by providing some necessary conditions for a partial applicative structure to be complete.
\end{abstract}
\section{Introduction}
Partial applicative structures (PAS for short) are very elementary mathematical entities. They just consist of a set $\mathsf{R}$ on which is defined a partial binary operation $\cdot:\mathsf{R}\times \mathsf{R}\rightharpoonup \mathsf{R}$. Such a structure can be seen at least from three different points of view. First, it can be seen simply as an algebraic structure which is a generalization of the algebraic structure called magma (in fact it is its ``partial'' version). We will call this point of view \emph{algebraic}. However, such a structure can also be thought of as carrying an abstract notion of computability; indeed, every element $r$ of $\mathsf{R}$ represents a partial unary function from $\mathsf{R}$ to itself which can be thought as ``computable'' according to $\mathcal{R}$. We will call this point of view \emph{computational}.
Third, one can consider the elements of $\mathsf{R}$ as ``realizers'' and subsets of $\mathsf{R}$ can be seen as ``propositions'' of which the elements are their ``realizers'' (or, if you prefer, ``constructive proofs''). This point of view gets richer when one allows propositions to depend on arbitrary elements of a set, that is if one considers the sets $\mathcal{P}(\mathsf{R})^I$ where $I$ is a set, and when an adequate notion of ``entailment'' between ``propositions'' is introduced. We could call this point of view \emph{logical}. 

Particular kinds of partial applicative structures are studied in different areas of mathematics (and this is not surprising given their extreme generality) according to at least one of the points of view just illustrated: monoids and groups are examples of partial applicative structures (\emph{algebraic} point of view), partial combinatory algebras (PCA) (see \cite{VOO08}) are studied as models of computation (\emph{computational} point of view) or as tools to produce indexed Heyting prealgebras (\emph{logical} point of view) called \emph{triposes} (see \cite{HJP}) which are used to produce elementary toposes (and in particular those toposes which are known under the name of \emph{realizability toposes}, like Hyland's Effective Topos in \cite{HYL}).  

In this paper, moved by mathematical taxonomy motivations, we aim to explore a possibility which is precluded in the classes of partial applicative structures mentioned above (because it simply leads to triviality in those cases), that is the possibility for a PAS to give rise to an indexed poset of ``dependent propositions''. This request is formulated in the framework of the logical point of view on PASs, but we will explore its connections with algebraic and computational properties. One could argue that it would be sufficient to characterize PASs giving rise to indexed preorders of ``dependent propositions'', since a posetal reflection can easily be performed subsequently. Of course, this is true, but it is not our goal: we want to understand when this happens directly. In Section 2 we will briefly introduce a list of properties which are computationally or algebraically relevant and recall some simple mutual relations between them. In Section 3 we will define precisely the indexed binary relation associated to a PAS, while in Section 4 we will provide characterizations for this indexed binary relation to satisfy a list of elementary properties which are often required to binary relations; the section will end with a characterization of those PASs which give rise to an indexed preorder, which we will call \emph{preorderal}, providing a common generalization both of the notion of monoid and of PCA. In Section 5 we will consider antisymmetry of the indexed binary relation induced by a PAS in relation with some particular classes of partial unary functions, under no hypotheses, under transitivity hypothesis and under totality hypothesis. We will conclude the section by providing a characterization of those PASs giving rise to indexed posets (we will call these PASs \emph{posetal}) in terms of elementary algebraic and computational properties. In Section 6 we will discuss the relation between the elementary properties introduced in the previous sections in the context of posetal PASs. In Section 7 we will study posetal PASs generated by a particular kind of partial functions. In Section 8 we will study preorderal PASs in relation with suprema and we will provide necessary conditions for a preorderal PAS to give rise to a complete indexed preorder. 

We will informally work in the classical set theory $\mathbf{ZFC}$.

\section{Partial applicative structures}
A \emph{partial applicative structure} (for short PAS) is a pair $\mathcal{R}=(\mathsf{R}, \cdot)$ where $\mathsf{R}$ is a non-empty set and $\cdot:\mathsf{R}\times \mathsf{R}\rightharpoonup \mathsf{R}$ is a partial function (we write $a\cdot b$ instead of $\cdot(a,b)$). An \emph{applicative closed term} is a string obtained as follows: every $a\in \mathsf{R}$ is an applicative term and if $t,s$ are applicative terms, then $(t\cdot s)$ is an applicative term. If $t$ is an applicative term we write $t\downarrow$ if the value of $t$ can be computed. If $t$ and $s$ are applicative terms, we write $t=s$ to mean that $t\downarrow$, $s\downarrow$ and the results of the two computations coincide. We write $t\simeq s$ to mean that $t\downarrow$ if and only if $s\downarrow$ and in that case the results of the two computations coincide. We will adopt the convention of association to the left for the operation $\cdot$, that is $a\cdot b\cdot c$ will mean $(a\cdot b)\cdot c$. We will usually remove external brackets from terms.

We will say that a PAS $\mathcal{R}$ is \emph{finite} if $\left|\mathsf{R}\right|$ is finite, and we will say that $\mathcal{R}$ is \emph{total} if the domain of $\cdot$ is $\mathsf{R}\times \mathsf{R}$. A \emph{trivial} PAS is a PAS $\mathcal{R}$ with $\left|\mathsf{R}\right|=1$ and $a\cdot a=a$ for the unique element $a$ of $\mathsf{R}$.
We will denote with $\mathbf{Finite}$, $\mathbf{Total}$ and $\mathbf{Trivial}$, respectively, the properties we have just introduced.

\subsection{Partial endofunctions represented by a PAS}
To every partial applicative structure $\mathcal{R}$ can be associated a function 
$$[-]:\mathsf{R}\rightarrow \mathsf{Part}(\mathsf{R},\mathsf{R})$$ where $\mathsf{Part}(\mathsf{R},\mathsf{R})$ is the set of partial functions from $\mathsf{R}$ to itself. Indeed, to each $r\in \mathsf{R}$ one can associate the partial function $[r]$ whose domain is the set of those $a\in \mathsf{R}$ for which $r\cdot a\downarrow$ and defined as $[r](a):=r\cdot a$ for every such $a$. We say that a partial function $f:\mathsf{R}\rightharpoonup\mathsf{R}$ is \emph{representable} in $\mathcal{R}$ if and only if there exists $r\in \mathsf{R}$ such that $[r]=f$ (and in this case we say that $r$ \emph{represents} $f$).
Notice that trivially a PAS $\mathcal{R}$ is total if and only if $[r]$ is total for every $r\in \mathsf{R}$.

We present here a list of properties that one can require to a partial applicative structure $\mathcal{R}$ in order to represent some particular partial functions.

\begin{property}[{\bf 1-Total}] There exists $r\in \mathsf{R}$ such that $[r]$ is total, i.e.\ there is at least one representable total function.
\end{property}


\begin{property}[{\bf Const}] All constants are representable in $\mathcal{R}$, that is, for every $a\in \mathsf{R}$ there exists $c\in \mathsf{R}$ such that $c\cdot b=a$ for every $b\in \mathsf{R}$.
\end{property}
\begin{property}[{\bf Id}] There exists $\mathbf{i}\in \mathsf{R}$ such that $[\mathbf{i}]=\mathsf{id}_{\mathsf{R}}$, that is $\mathbf{i}\cdot a=a$ for every $a\in \mathsf{R}$.
\end{property}

\begin{property}[{\bf TM}]
For every $a,b\in \mathsf{R}$ there exists $r\in \mathsf{R}$ such that $r\cdot a=b$.\footnote{{\bf TM} stands for ``totally matching''.}
\end{property}

\begin{property}[{\bf 1W}]
For every $r,s,a,b\in \mathsf{R}$, if $r\cdot a=b$ and $s\cdot b=a$, then $a=b$.\footnote{{\bf 1W} stands for ``one way''.}
\end{property}

\begin{property}[{\bf Ext}] For every $r,s\in \mathsf{R}$, if $[r]=[s]$, then $r=s$.
\end{property}

\begin{property}[{\bf Comp}]
For every $r,s\in \mathsf{R}$ there exists $t\in \mathsf{R}$ such that $[s]\circ [r]= [t]$, that is $s\cdot(r\cdot a)\simeq t\cdot a$ for every $a\in \mathsf{R}$.
\end{property}

\begin{proposition}\label{basic}Let $\mathcal{R}$ be a PAS. Then, 
\begin{enumerate}
\item If $\mathcal{R}$ is trivial, then $\mathbf{Id}$, $\mathbf{Const}$, $\mathbf{Comp}$, $\mathbf{Ext}$ and $\mathbf{1W}$ hold;
\item $\mathbf{Const}\Rightarrow \mathbf{TM}$;
\item $\mathbf{Const}\vee \mathbf{Id}\vee \mathbf{Total}\Rightarrow \mathbf{1}\textrm{-}\mathbf{Total}$;
\item $\mathbf{Id}\wedge\mathbf{Comp}$ if and only if $\{[r]|\,r\in \mathsf{R}\}$ is a submonoid of the monoid of partial functions from $\mathsf{R}$ to $\mathsf{R}$;
\item $\mathbf{1W}\wedge \mathbf{TM}\Rightarrow \mathbf{Trivial}$.
\end{enumerate}
\end{proposition}
\begin{proof}
1.\ clearly holds. 2.\ holds since if $\mathbf{Const}$ holds, then for every $a,b\in \mathsf{R}$, there exists $c\in \mathsf{R}$ such that  $c\cdot x=b$ for every $x\in \mathsf{R}$, from which in particular it follows that $c\cdot a=b$. 3.\ and 4.\ hold by definition. For 5., if $\mathbf{1W}$ and $\mathbf{TM}$ hold, then for every $a,b$ there exist $r,s\in \mathsf{R}$ with $r\cdot a=b$ and $s\cdot b=a$ by $\mathbf{TM}$, and thus $a=b$ by $\mathbf{1W}$; moreover, since $\mathbf{TM}$ holds the only element $a$ of $\mathsf{R}$ must satisfy $a\cdot a=a$.
\end{proof}

\begin{proposition}\label{basicfin} Assume $\mathcal{R}$ to be a PAS satisfying $\mathbf{Finite}$. Then:
\begin{enumerate}
\item if $\mathcal{R}$ satisfies $\mathbf{Const}$, then for every $r,a,b\in \mathsf{R}$, $r\cdot a=r\cdot b$; 
\item $\mathbf{Const}\Rightarrow \mathbf{Comp}$;
\item $\mathbf{TM}\Rightarrow\mathbf{Ext}\wedge\mathbf{Total}$;
\item $\mathbf{Const}\wedge \mathbf{Id}\Rightarrow \mathbf{Trivial}$.
\end{enumerate}
\end{proposition}
\begin{proof} Let $\mathcal{R}$ be a PAS satisfying $\mathbf{Finite}$ with $\left|\mathsf{R}\right|=n$.
\begin{enumerate}
\item If $\mathcal{R}$ satisfies $\mathbf{Const}$, then one can define an injective function $a\mapsto k_a$ from $\mathsf{R}$ to itself sending each $a\in \mathsf{R}$ to an element $k_a\in \mathsf{R}$ representing the constant function with value $a$. This function is bijective since $\mathsf{R}$ is finite. Thus, every $r\in \mathsf{R}$ represents a constant function.
\item This follows from the proof of 1.\ and the fact that the composition of two constant functions is a constant function.
\item If $\mathcal{R}$ satisfies $\mathbf{TM}$, then each pair $(a,b)\in \mathsf{R}\times \mathsf{R}$ must be element of $[r]$ for some $r\in \mathsf{R}$. Since such pairs are $n^2$, it follows that $\left|\bigcup_{r\in \mathsf{R}} [r]\right|\geq n^2$. Since $\left|[r]\right|\leq n$ for every $r\in \mathsf{R}$, it follows that $\left|[r]\right|= n$ for every $r\in \mathsf{R}$ and that $[r]\cap [s]=\emptyset$ for every $r,s\in \mathsf{R}$ with $r\neq s$. Thus, $[r]$ is total for every $r\in \mathsf{R}$ and if $[r]=[s]$, then $r=s$.
\item Assume $\mathcal{R}$ to satisfy $\mathbf{Id}$ and $\mathbf{Const}$ and let $\mathbf{i}$ be an element of $\mathsf{R}$ such that $[\mathbf{i}]=\mathsf{id}_{\mathsf{R}}$; then from 1.\ it follows that $a=\mathbf{i}\cdot a=\mathbf{i}\cdot b=b$ for every $a,b\in \mathsf{R}$; moreover, $\mathbf{i}\cdot \mathbf{i}=\mathbf{i}$. Thus $\mathcal{R}$ is trivial.
\end{enumerate}
\end{proof}

\subsection{Standard algebraic properties}
The following properties are often required to algebraic structures. 
\begin{property}[$\mathbf{ID}_{R}$] There exists $\mathbf{j}$ such that $a\cdot \mathbf{j}=a$ for every $a\in  \mathsf{R}$.
\end{property}

We do not need to add a new property requiring the existence of a left identity for $\cdot$, since $\mathbf{Id}$ already states such a property.

Associativity and abelianity are formulated in the partial context as follows:

\begin{property}[{\bf Assoc}] 
$a\cdot(b\cdot c)\simeq (a\cdot b)\cdot c$ for every $a,b,c\in \mathsf{R}$.
\end{property}

\begin{property}[{\bf Ab}]
$a\cdot b\simeq b\cdot a$ for every $a,b\in \mathsf{R}$.
\end{property}
Obviously we have that $\mathbf{Id}\wedge \mathbf{Ab}\Leftrightarrow \mathbf{ID}_{R}\wedge \mathbf{Ab}$. 
In the following proposition we illustrate some interactions between some properties in the previous section and those in this one.
\begin{proposition}\label{interact} Let $\mathcal{R}$ be a PAS. Then:
\begin{enumerate}
\item $\mathbf{ID}_{R}\Rightarrow \mathbf{Ext}$;
\item If $\mathbf{Total}\wedge \mathbf{Assoc}$, then $[s\cdot r]=[s]\circ [r]$ for every $r,s\in \mathsf{R}$ and, in particular, $\mathbf{Comp}$ holds.
\end{enumerate}
\end{proposition}
\begin{proof} Let $\mathcal{R}$ be a PAS. Assume $\mathbf{ID}_{R}$ to hold and $r\cdot a\simeq s\cdot a$ for every $a\in \mathsf{R}$; then in particular $r=r\cdot \mathbf{j}=s\cdot \mathbf{j}=s$. Thus we have proved 1. 

If a PAS $\mathcal{R}$ safisfies $\mathbf{Total}\wedge \mathbf{Assoc}$, then $s\cdot r\downarrow $ for every $r,s\in \mathsf{R}$ and, since $(s\cdot r)\cdot a= s\cdot (r\cdot a)$ for every $a\in \mathsf{R}$, $[s\cdot r]=[s]\circ [r]$. In particular, $\mathbf{Comp}$ holds. This proves 2.
\end{proof}

We recall that $\mathcal{R}$ is a monoid if and only if it satisfies $\mathbf{Total}$, $\mathbf{Assoc}$, $\mathbf{ID}_R$ and $\mathbf{Id}$.
Obviously, every trivial PAS is an abelian monoid, that is a monoid satisfying $\mathbf{Ab}$.

\subsection{Combinators $\mathbf{k}$ and $\mathbf{s}$.}
We consider now two properties concerning the existence of certain combinators.
\begin{property}[{\bf K}] There exists $\mathbf{k}\in \mathsf{R}$ such that $\mathbf{k}\cdot a\cdot b=a$ for every $a,b\in \mathsf{R}$.
\end{property}
\begin{proposition}\label{Konst} Let $\mathcal{R}$ be a PAS. Then $\mathbf{K}\Rightarrow \mathbf{Const}$. 
\end{proposition}
\begin{proof}
If $\mathcal{R}$ satisfies $\mathbf{K}$ and $a\in \mathsf{R}$, then $\mathbf{k}\cdot a$ represents the constant function with value $a$.
\end{proof}
\begin{proposition}\label{Ktriv}
Let $\mathcal{R}$ be a PAS. Then
$$\mathbf{Trivial}\Leftrightarrow \mathbf{K}\wedge \mathbf{Finite}\Leftrightarrow \mathbf{K}\wedge \mathbf{ID}_{R}\Leftrightarrow \mathbf{K}\wedge \mathbf{Assoc}\Leftrightarrow \mathbf{K}\wedge \mathbf{Ab}$$
\end{proposition}
\begin{proof}Each trivial PAS is a finite abelian monoid which trivially satisfies $\mathbf{K}$, thus we only need to prove that a PAS $\mathcal{R}$ satisfying $\mathbf{K}$ is trivial if it satisfies also one among $\mathbf{Finite}$, $\mathbf{ID}_{R}$, $\mathbf{Assoc}$ or $\mathbf{Ab}$. Since $\mathbf{k}\cdot \mathbf{k}\downarrow$, it will be sufficient to show that all elements of $\mathcal{R}$ are equal in each of the four cases.

If $\mathcal{R}$ satisfies one of these last two properties, then for every $a\in \mathsf{R}$ we have $\mathbf{k}=\mathbf{k}\cdot \mathbf{k}\cdot \mathbf{k}=\mathbf{k}\cdot(\mathbf{k}\cdot \mathbf{k})\cdot a \cdot \mathbf{k}=(\mathbf{k}\cdot \mathbf{k})\cdot \mathbf{k}\cdot a \cdot \mathbf{k}=\mathbf{k}\cdot a\cdot \mathbf{k}=a$. 

If $\mathbf{ID}_{R}$ holds, then $a=\mathbf{k}\cdot a\cdot \mathbf{k}=\mathbf{k}\cdot \mathbf{j}\cdot a\cdot \mathbf{k}=\mathbf{j}\cdot \mathbf{k}$ for every $a\in \mathsf{R}$.

Finally, if $\mathcal{R}$ is finite, then $\mathbf{k}\cdot a=\mathbf{k}\cdot b$ for every $a,b\in\mathsf{R}$ by  Propositions \ref{Konst} and \ref{basicfin}(1). Thus $a=\mathbf{k} \cdot a \cdot \mathbf{k}=\mathbf{k} \cdot b \cdot \mathbf{k}=b$ for every $a,b\in \mathsf{R}$.
\end{proof}
\begin{property}[{\bf S}]
There exists $\mathbf{s}\in \mathsf{R}$ such that $\mathbf{s}\cdot a\cdot b\downarrow$ for every $a,b\in \mathsf{R}$ and $\mathbf{s}\cdot a\cdot b\cdot c\simeq a\cdot c\cdot (b\cdot c)$ for every  $a,b,c\in \mathsf{R}$.
\end{property}
Notice that $\mathbf{S}$ implies {\bf 1-Total} by definition, since $\mathbf{s}$ represents a total function. 

Partial combinatory algebras ($\mathbf{PCA}$) are very well-known structures arising from computability theory. They are PASs satisfying $\mathbf{K}$ and $\mathbf{S}$. The trivial PAS is clearly a PCA. 

The following fact is very well-known:

\begin{proposition}\label{PCAthm}
Let $\mathcal{R}$ be a PCA. Then $\mathbf{Id}$ and $\mathbf{Comp}$ hold.
\end{proposition}
\begin{proof} A left identity is provided by $\mathbf{s}\cdot \mathbf{k}\cdot \mathbf{k}$. For $\mathbf{Comp}$, if $r,s\in \mathsf{R}$, then $[t]=[s]\circ [r]$ for $t:=\mathbf{s}\cdot (\mathbf{k}\cdot s)\cdot (\mathbf{s}\cdot (\mathbf{k}\cdot r)\cdot (\mathbf{s}\cdot \mathbf{k}\cdot \mathbf{k}))$.
\end{proof}
We also have this proposition which follows directly from Propositions \ref{Konst}, \ref{basic}(2,5) and \ref{Ktriv}.
\begin{proposition}\label{PCAtriv}
Let $\mathcal{R}$ be a PCA. Then:
$$\mathbf{Trivial}\Leftrightarrow\mathbf{Ab}
\Leftrightarrow\mathbf{Assoc}
\Leftrightarrow\mathbf{ID}_{R}
\Leftrightarrow\mathbf{1W}\Leftrightarrow \mathbf{Finite}$$
\end{proposition}

\subsection{Summarizing}
The relations we established between the properties we considered up to now can be summarized in the following diagram.
{\tiny
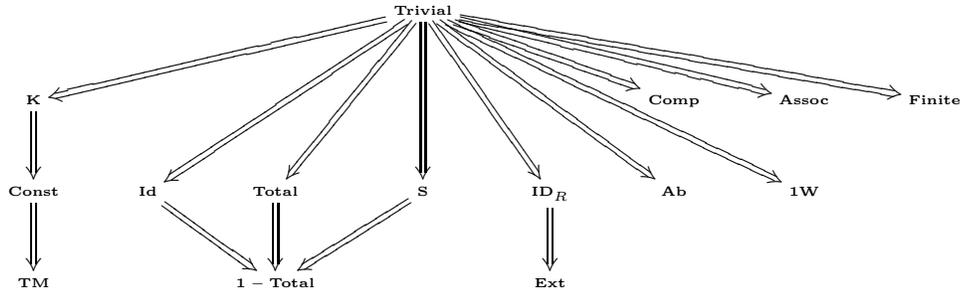
$$\xymatrix{
							&			&			&\mathbf{Trivial}\ar@{=>}[dlll]\ar@{=>}[ddll]\ar@{=>}[ddl]\ar@{=>}[dd]\ar@{=>}[ddr]\ar@{=>}[drr]\ar@{=>}[ddrr]	\ar@{=>}[drrr]\ar@{=>}[ddrrr]\ar@{=>}[drrrr]				&	&	&	&\\
\mathbf{K}\ar@{=>}[d]			&			&			&	&	&\mathbf{Comp}		&\mathbf{Assoc}	&\mathbf{Finite}&	&	&\\
\mathbf{Const}\ar@{=>}[d]		&\mathbf{Id}\ar@{=>}[rd]		&\mathbf{Total}	\ar@{=>}[d]	&\mathbf{S}\ar@{=>}[ld]	&\mathbf{ID}_{R}	\ar@{=>}[d]		&\mathbf{Ab}	&\mathbf{1W}	\\
\mathbf{TM}					&				&\mathbf{1}-\mathbf{Total}				&					&\mathbf{Ext}\\
}$$
}
while in the finite case the situation is the following
{\tiny
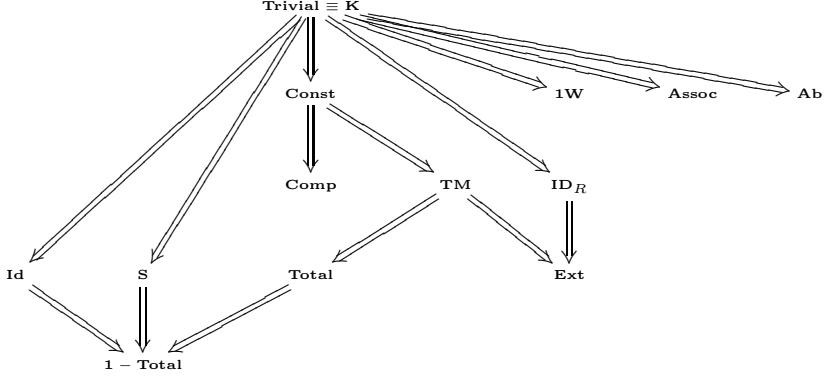
$$\xymatrix{
			&					&\mathbf{Trivial}\equiv\mathbf{K}\ar@{=>}[d]\ar@{=>}[dddll]\ar@{=>}[dddl]\ar@{=>}[ddrr]	\ar@{=>}[drr]\ar@{=>}[drrr]\ar@{=>}[drrrr]		&			&				\\		
			&					&\mathbf{Const}\ar@{=>}[d]\ar@{=>}[dr]		&							&\mathbf{1W}	&\mathbf{Assoc}	&\mathbf{Ab}	\\
			&					&\mathbf{Comp	} &\mathbf{TM}\ar@{=>}[dl]\ar@{=>}[dr]			&\mathbf{ID}_{R}\ar@{=>}[d]		&	&	\\			
\mathbf{Id}\ar@{=>}[dr]	&\mathbf{S}\ar@{=>}[d]			&\mathbf{Total}\ar@{=>}[dl]		&	&\mathbf{Ext}		&		&		&		\\	
			&\mathbf{1}-\mathbf{Total}	\\
}
$$
}

\section{Indexed relations arising from PASs}

Let us fix a PAS $\mathcal{R}=(\mathsf{R},\cdot)$. If $A$ and $B$ are subsets of $\mathsf{R}$, we can define the implication set 
$$A\Rightarrow B:=\{r\in \mathsf{R}|\,\forall a\in A\,(r\cdot a\downarrow\,\wedge\; r\cdot a\in B)\}$$

We can also define a relation $\vdash$ on $\mathcal{P}(\mathsf{R})$ as follows: 
$$A\vdash B\textnormal{ if and only if }A\Rightarrow B\neq \emptyset$$ Moreover, for every set $I$, we can define a binary relation $\vdash_{I}$ on $\mathcal{P}(\mathsf{R})^{I}$ as follows: for $\varphi, \psi: I\rightarrow \mathcal{P}(\mathsf{R})$, 
$$\varphi\vdash_{I} \psi\textrm{ if and only if }\bigcap_{i\in I}(\varphi(i)\Rightarrow \psi(i))\neq \emptyset.$$  
We write $r\Vdash (\varphi\vdash_{I}\psi)$ whenever $r\in \bigcap_{i\in I}(\varphi(i)\Rightarrow \psi(i))$. Notice that, as a direct consequence of the definition, if $\varphi\vdash_{I}\psi$, then $\varphi(i)\vdash\psi(i)$ for every $i\in I$; however the opposite does not hold in general.

Let us introduce a category $\mathbf{Bin}$ defined as follows: its objects are pairs $(A,\rho)$ in which $\rho$ is a binary relation on the set $A$, while an arrow in $\mathbf{Bin}$ from $(A,\rho)$ to $(B,\sigma)$ is a function $f:A\rightarrow B$ such that, for every $a,a'\in A$, $\rho(a,a')$ implies $\sigma(f(a),f(a'))$ (the composition is defined to be the usual set-theoretical composition of functions). 

The assignments $I\mapsto (\mathcal{P}(\mathsf{R})^{I},\vdash_{I})$ and $f\mapsto (-)\circ f$ 
define a $\mathbf{Set}$-indexed binary relation, that is a contravariant functor from $\mathbf{Set}$ to $\mathbf{Bin}$ which we denote with 
$$\pi_{\mathcal{R}}:\mathbf{Set}^{op}\rightarrow \mathbf{Bin}$$
This functor is well defined since if $f$ is a function from $J$ to $I$ and $\varphi,\psi\in \pi_{\mathcal{R}}(I)$ then 
$$\bigcap_{i\in I}(\varphi(i)\Rightarrow \psi(i))\subseteq\bigcap_{j\in J}(\varphi(f(j))\Rightarrow \psi(f(j))) $$
For every set $I$, $\pi_{\mathcal{R}}(I)=(\mathcal{P}(\mathsf{R})^{I},\vdash_{I})$ is called the \emph{fiber} of $\pi_{\mathcal{R}}$ over $I$; for every function $f$, we call $\pi_{\mathcal{R}}(f)$ the \emph{reindexing map} along $f$.

One can also notice that $(\mathcal{P}(\mathsf{R}),\vdash)$ and $(\mathcal{P}(\mathsf{R})^{\{a\}},\vdash_{\{a\}} )$ are isomorphic objects in $\mathbf{Bin}$ for every $a$, via the map sending each subset $A$ of $\mathsf{R}$ to the function $\{(a,A)\}$.

When $\mathcal{R}$ is a PCA the functor $\pi_{\mathcal{R}}$ factorizes through the subcategory of $\mathbf{Bin}$ of Heyting prealgebras (see e.g.\ \cite{VOO08}).
However, in the general case we are far from such a nice situation. In the next section we first want to show necessary and sufficient conditions for every fiber of $\pi_{\mathcal{R}}$ to satisfy certain elementary properties and then we will characterize those partial applicative structures giving rise to indexed preorders.

\section{Preorderal partial applicative structures}
\subsection{Elementary properties of indexed relations}

We start this section by recalling some elementary properties which a binary relation can satisfy. 

A binary relation $R$ on a set $A$  
\begin{enumerate}
    \item is \emph{reflexive} if $R(a,a)$ holds for every $a\in A$;
    \item is \emph{transitive} if, for every every $a,b,c\in A$, if $R(a,b)$ and $R(b,c)$ hold, then $R(a,c)$ holds;
    \item is \emph{symmetric} if, for every $a,b\in A$, if $R(a,b)$ holds, then $R(b,a)$ holds;
    \item is \emph{antisymmetric} if, for every $a,b\in A$, if $R(a,b)$ and $R(b,a)$ hold, then $a=b$;
    \item has \emph{maximum} if there exists $a\in A$ such that $R(b,a)$ holds for every $b\in A$;
    \item has \emph{minimum} if there exists $a\in A$ such that $R(a,b)$ holds for every $b\in A$;
    \item has \emph{right bounds} if for every $b\in A$ there is $a\in A$ such that $R(b,a)$ holds;
    \item has \emph{left bounds} if for every $b\in A$ there is $a\in A$ such that $R(a,b)$ holds.
\end{enumerate}
We say that an object $(A,\rho)$ of $\mathbf{Bin}$ \emph{has} one of the previous properties if the relation $\rho$ on $A$ has that property, and we say that a $\mathbf{Set}$-indexed binary relation \emph{has} one of the previous properties if all its fibers have it. If $\mathcal{R}$ is a PAS, we say that it \emph{has} one of the previous properties if $\pi_{\mathcal{R}}$ has it. 

\begin{proposition}\label{ref}
A partial applicative structure $\mathcal{R}$ is reflexive if and only if $\mathbf{Id}$ holds.
\end{proposition}
 \begin{proof}
Suppose $\mathcal{R}$ is reflexive. Let $\mathsf{sgl}:\mathsf{R}\rightarrow \mathcal{P}(\mathsf{R})$ be the singleton function sending each $a\in \mathsf{R}$ to $\{a\}$. Since $\vdash_{\mathsf{R}}$ is reflexive, then $\mathsf{sgl}\vdash_{\mathsf{R}}\mathsf{sgl}$. This means that there exists $\mathbf{i}\in \mathsf{R}$ such that $\mathbf{i}\cdot a=a$ for every $a\in \mathsf{R}$. Conversely, if there exists such an element $\mathbf{i}$, then for every set $I$ and every $\varphi:I\rightarrow \mathcal{P}(\mathsf{R})$, we have that $\mathbf{i}\in \varphi(j)\Rightarrow \varphi(j)$ for every $j \in I$, that is $\mathbf{i}\Vdash(\varphi\vdash_{I}\varphi)$.
\end{proof}

\begin{remark}\label{ext}
Notice that if $\mathcal{R}$ is reflexive, the order $\vdash$ extends the inclusion $\subseteq$ on $\mathcal{P}(\mathsf{R})$, that is $A\subseteq B$ implies $A\vdash B$. The converse implication does not hold. Indeed, one can consider the PAS $(\mathbb{N},\cdot)$ where
$$\begin{cases}
0\cdot m=m+1\textrm{ for every }m\in \mathbb{N}\\
(n+1)\cdot m=m\textrm{ for every }n,m\in \mathbb{N}\textrm{ with }n\neq m\\
(n+1)\cdot n\not\downarrow\textrm{ for every }n\in \mathbb{N}\\
\end{cases}$$
For $A\subseteq B$ we can distinguish two cases
\begin{enumerate}
\item $A=B=\mathbb{N}$, and in this case $0\in A\Rightarrow B$;
\item there exists $n\in \mathbb{N}$ such that $n\notin A$. In this case, $n+1\in A\Rightarrow B$.
\end{enumerate}
\end{remark}
\noindent In any case, $A\vdash B$. However, $\mathsf{id}_{\mathbb{N}}$ is not represented by any natural number, thus $\mathcal{R}$ is not reflexive.

 \begin{proposition}\label{tra}
A partial applicative structure $\mathcal{R}$ is transitive if and only if for every $r,s\in \mathsf{R}$, there exists $t\in \mathsf{R}$ such that for every $a\in \mathsf{R}$, if $s\cdot (r\cdot a)\downarrow$, then $t\cdot a\downarrow $ and $s\cdot (r\cdot a)=t\cdot a$, that is $[t]\supseteq [s]\circ [r]$.
\end{proposition}
\begin{proof} Let us assume $\mathcal{R}$ to be transitive. Then, 
for every $r,s\in \mathsf{R}$ we can consider the set $\mathsf{R}_{r,s}:=\{(a,b,c)\in \mathsf{R}\times \mathsf{R}\times \mathsf{R}|\;r\cdot a=b\,\wedge\,s\cdot b=c\}$.
Let $\mathsf{sgl}$ be the singleton function defined in the proof of Proposition \ref{ref}.
Since $r\Vdash \left(\mathsf{sgl}\circ\pi_{1}\vdash_{\mathsf{R}_{r,s}}\mathsf{sgl}\circ\pi_{2}\right)$ and  $s\Vdash \left(\mathsf{sgl}\circ\pi_{2}\vdash_{\mathsf{R}_{r,s}}\mathsf{sgl}\circ\pi_{3}\right)$ by definition of $\mathsf{R}_{r,s}$, and $\vdash_{\mathsf{R}_{r,s}}$ is transitive, there exists $t\in \mathsf{R}$ such that $t\Vdash \left( \mathsf{sgl}\circ\pi_{1}\vdash_{\mathsf{R}_{r,s}}\mathsf{sgl}\circ\pi_{3}\right)$. In particular, $t$ satisfies the following property: for every $a\in \mathsf{R}$, if $s\cdot (r\cdot a)\downarrow$, then $t\cdot a\downarrow $ and $s\cdot (r\cdot a)=t\cdot a$.
Conversely, let us assume that $\mathcal{R}$ satisfies the condition in the statement of the proposition. If $r\Vdash (\varphi\vdash_I \psi)$ and $s\Vdash (\psi\vdash_I \rho)$, there exists $t\in\mathsf{R}$ such that $[t]\supseteq  [s]\circ [r]$ and hence in particular such that $t\Vdash (\varphi\vdash_I \rho)$.
\end{proof}
\begin{proposition}\label{sym}
There is no symmetric partial applicative structure $\mathcal{R}$.\end{proposition}
\begin{proof}
Assume $\mathcal{R}$ is symmetric, then $(\mathcal{P}(\mathcal{R}),\vdash)$ is symmetric. 
Since $\emptyset\vdash \mathsf{R}$, then by symmetry $\mathsf{R}\vdash \emptyset$ from which it follows that $\mathsf{R}=\emptyset$, resulting in a contradiction. \end{proof}
\begin{proposition}\label{ant}
A partial applicative structure $\mathcal{R}$ is antisymmetric if and only if $(\mathcal{P}(\mathsf{R}),\vdash)$ is antisymmetric.
\end{proposition}
\begin{proof}
If $\mathcal{R}$ is antisymmetric, then $\pi_{\mathcal{R}}(\{0\})$ is antisymmetric, from which it follows that $(\mathcal{P}(\mathsf{R}), \vdash)$ is antisymmetric.

Conversely, suppose $(\mathcal{P}(\mathsf{R}),\vdash)$ is antisymmetric. If $\varphi\vdash_{I}\psi$ and $\psi\vdash_{I}\varphi$, then $\varphi(i)\vdash\psi(i)$ and $\psi(i)\vdash\varphi(i)$ for every $i\in I$, from which, by the antisymmetry of $\vdash$, $\varphi(i)=\psi(i)$ for every $i\in I$. Thus $\varphi=\psi$.\end{proof}
\begin{proposition}\label{max}
For a partial applicative structure $\mathcal{R}$ the following are equivalent:
\begin{enumerate}
\item $\mathcal{R}$ has right bounds;
\item $\mathcal{R}$ has maximum;
\item $\mathcal{R}$ satisfies {\bf 1-Total}.
\end{enumerate}
\end{proposition}
\begin{proof}It is enough to prove the following implications.
\begin{enumerate}
\item[$1.\Rightarrow 3.$]If $\mathcal{R}$ has right bounds, then $\pi_{\mathcal{R}}(\{0\})$ has right bounds, from which it follows that $(\mathcal{P}(\mathsf{R}),\vdash)$ has right bounds. In particular, there must exists $A\subseteq \mathsf{R}$ such that $\mathsf{R}\vdash A$. From this it follows that there exists $r\in \mathsf{R}$ such that $[r]$ is total. 
\item[$3.\Rightarrow 2.$] If $r$ represents a total function, then for every set $I$ the constant function $\mathsf{R}_{I}:I\rightarrow \mathcal{P}(\mathsf{R})$ with value $\mathsf{R}$ is a maximum in $\pi_{\mathcal{R}}(I)$, since $r\Vdash (\varphi\vdash_I \mathsf{R}_{I})$ for every $\varphi\in \pi_{\mathcal{R}}(I)$.
\item[$2.\Rightarrow 1.$] Trivial, by definition.
\end{enumerate}
\end{proof}
\begin{corollary}\label{maxnomax}
Every reflexive PAS has maximum.\end{corollary}
\begin{proposition}
Every PAS has minimum. Hence in particular every PAS has left bounds.
\end{proposition}
\begin{proof}
For every $I$, a minimum for $\pi_\mathcal{R}(I)$ is the constant function $\emptyset_I:I\rightarrow \mathcal{P}(\mathsf{R})$ with value $\emptyset$.
\end{proof}

\begin{remark} Notice that minima are preserved by $\pi_\mathcal{R}(f)$ for every function $f$ and every PAS $\mathcal{R}$, since every fiber $\pi_{\mathcal{R}}(I)$ has only one minimum, that is $\emptyset_I$, and $\emptyset_I\circ f=\emptyset_J$ for every function $f:J\rightarrow I$. Moreover, if $\mathcal{R}$ has maximum, then $\pi_\mathcal{R}(f)$ preserves maxima for every function $f$. This follows from the fact that $\top_I$ is a maximum in $\pi_{\mathcal{R}}(I)$ if and only if $\mathsf{R}_I\vdash_I \top_I$, and from the fact that $\mathsf{R}_I\circ f=\mathsf{R}_J$ for every function $f:J\rightarrow I$.
\end{remark}
We introduce the abbreviations $\mathbf{Ant}$ and $\mathbf{Trans}$ for the property of $\mathcal{R}$ being antisymmetric and transitive, respectively.  

Before proceeding, we will go back to the previously introduced properties in order to see whether they have an intepretation in terms of the indexed relations induced by $\mathcal{R}$. As we have already seen, {\bf 1-Total} is equivalent to requiring the fibers of $\pi_{\mathcal{R}}$ to have maxima, while $\mathbf{Id}$ is equivalent to the reflexivity of $\pi_{\mathcal{R}}$. Moreover we have the following.

\begin{proposition} A PAS $\mathcal{R}$ satisfies $\mathbf{Const}$ if and only if $A\vdash B$ for every pair of non-empty subsets $A,B$ of $\mathsf{R}$.
\end{proposition}
\begin{proof}
If $\mathbf{Const}$ holds, $A\subseteq \mathsf{R}$ and $b\in B\subseteq \mathsf{R}$, then we can find $c\in \mathsf{R}$ such that $c\cdot x=b$ for every $x\in \mathsf{R}$; in particular $c\in A\Rightarrow B$, from which it follows that $A\vdash B$. Conversely, if $A\vdash B$ for every pair of non-empty subsets $A,B$ of $\mathsf{R}$ , then $\mathsf{R}\vdash \{a\}$ for every $a\in \mathsf{R}$, from which it follows that there exists $c\in \mathsf{R}$ such that $c\cdot x=a$ for every $x\in \mathsf{R}$; thus $\mathbf{Const}$ holds.
\end{proof}

\begin{proposition} A PAS $\mathcal{R}$ satisfies $\mathbf{TM}$ if and only if $\{a\}\vdash \{b\}$ for every pair of elements $a,b$ of $\mathsf{R}$.
\end{proposition}
\begin{proof}
$\mathbf{TM}$ holds if and only if for every $a$ and $b$ there exists $r$ such that $r\cdot a=b$ if and only if $\{a\}\vdash \{b\}$ for every $a$ and $b$.
\end{proof}
\begin{proposition}\label{car1w}
A PAS satisfies $\mathbf{1W}$ if and only if $\vdash$ is antisymmetric restricted to singletons of $\mathsf{R}$.
\end{proposition}
\begin{proof}
This is an immediate consequence of the definition of $\mathbf{1W}$.
\end{proof}
\begin{proposition}\label{compant}
Let $\mathcal{R}$ be a PAS, then $\mathbf{Comp}\Rightarrow \mathbf{Trans}$.
\end{proposition}
\begin{proof}
This is a consequence of Proposition \ref{tra}. \end{proof}

As we have seen $\mathbf{Comp}$ together with $\mathbf{Id}$ implies that the partial functions represented by $\mathcal{R}$ form a submonoid of the monoid of partial functions. In the case in which $\mathbf{Trans}$ holds in place of $\mathbf{Comp}$ this is no longer true in general, however we have the following:
\begin{proposition}If $\mathcal{R}$ satisfies $\mathbf{Trans}$ and $\mathbf{Id}$, then $$\{f\in \mathsf{R}^{\mathsf{R}}|\,\exists\,r\in \mathsf{R}\,(f=[r])\}$$ is a submonoid of the monoid $(\mathsf{R}^{\mathsf{R}},\circ)$.
\end{proposition}
\subsection{Preorderal partial applicative structures}
\begin{definition}
A partial applicative structure $\mathcal{R}$ is \emph{preorderal} if $\pi_{\mathcal{R}}$ factors through the subcategory $\mathbf{Pord}$ of $\mathbf{Bin}$ of preordered sets, that is for every set $I$, $\vdash_{I}$ is a preorder on $\mathcal{P}(\mathsf{R})^{I}$.
\end{definition}

Putting together Propositions \ref{ref} and \ref{tra} we obtain the following
\begin{theorem}
A partial applicative structure $\mathcal{R}$ is preorderal if and only if 
\begin{enumerate}
\item there exists $\mathbf{i}\in \mathsf{R}$ such that $\mathbf{i}\cdot a=a$ for every $a\in \mathsf{R}$ and
\item for every $r,s\in \mathsf{R}$, there exists $t\in \mathsf{R}$ such that for every $a\in A$, if $s\cdot (r\cdot a)\downarrow$, then $t\cdot a\downarrow $ and $s\cdot (r\cdot a)=t\cdot a$.
\end{enumerate}
\end{theorem}
The theorem above shows that $\mathcal{R}$ is preorderal if and only if $\{[r]|\,r\in \mathsf{R}\}$ is an untyped lax computability model according to \cite{LN}.

A nice remark is that the notion of preorderal PAS turns out to be a generalization of the two orthogonal notions of monoid and of PCA.
\begin{proposition} The following hold:
\begin{enumerate}
\item If $\mathcal{R}$ is a monoid and a PCA, then it is trivial;
\item If $\mathcal{R}$ is a monoid, then $\mathcal{R}$ is preorderal;
\item If $\mathcal{R}$ is a PCA, then $\mathcal{R}$ is preorderal.
\end{enumerate}
\end{proposition}
\begin{proof}
1.\ is a consequence of Proposition \ref{PCAtriv}, 2.\ a consequence of Proposition \ref{interact}(2) and \ref{compant}, and 3.\ a consequence of Propositions \ref{PCAthm} and \ref{compant}.
\end{proof}

\section{Antisymmetric partial applicative structures}
We have already characterized antisymmetric partial applicative structures in Proposition \ref{ant}. However, we would like to obtain a different, more meaningful and usable, characterization in terms of the properties of the elements of $\mathsf{R}$, of the partial functions they represent and of the operation $\cdot$.

\subsection{Antisymmetry and two kinds of partial functions}
\begin{definition}
Let $f:\mathsf{R}\rightharpoonup \mathsf{R}$ be a partial function defined on a non-empty set $\mathsf{R}$. For every $a\in \mathsf{R}$ we define its \emph{order} $\mathsf{ord}_f(a)$ relative to $f$ as the minimum $n\in \mathbb{N}^{+}$ such that $f^n(a)=a$, if such an $n$ exists, and as $+\infty$ otherwise. Moreover, for every $a\in \mathsf{R}$ we define its \emph{divergence} $\delta_f(a)$ relative to $f$ as the minimum $n\in \mathbb{N}^{+}$ such that $f^{n}(a)\not\downarrow$, if such an $n$ exists, and as $+\infty$ otherwise. We define the \emph{divergence} $\Delta(f)$ of $f$ as the supremum of the set $$\{\delta_f(a)|\,a\in \mathsf{R}\,\wedge \,\delta_f(a)<+\infty\}$$
This supremum is taken to be $1$ if such a set is empty.

\end{definition}

Notice that if $f:\mathsf{R}\rightharpoonup \mathsf{R}$, $a\in \mathsf{R}$, $\mathsf{ord}_f(a)<+\infty$ and $0\leq i<j<\mathsf{ord}_f(a)$, then $f^{i}(a)\neq f^{j}(a)$.
Similarly, if $f:\mathsf{R}\rightharpoonup \mathsf{R}$, $a\in \mathsf{R}$, $\delta_f(a)<+\infty$ and $0\leq i<j<\delta_f(a)$, then $f^{i}(a)\neq f^{j}(a)$.

\begin{definition} A partial function $f:\mathsf{R}\rightharpoonup \mathsf{R}$ is TO if for every $a\in \mathsf{R}$ one has $\mathsf{ord}_f(a)=2k+1$ for some $k\in \mathbb{N}$ or $\delta_{f}(a)\in \mathbb{N}^{+}$. A partial function $f:\mathsf{R}\rightharpoonup \mathsf{R}$ is TL if for every $a\in \mathsf{R}$ one has $\mathsf{ord}_f(a)=1$ or $\delta_{f}(a)\in \mathbb{N}^{+}$.
\end{definition}

Letters T, L and O in TO and TL stay for \emph{trees}, \emph{loops} and \emph{odd cycles}, respectively. The reason of the names is evident if we represent a partial function $f:\mathsf{R}\rightharpoonup \mathsf{R}$ as a directed graph (possibly with loops) of which the vertices are the elements of $\mathsf{R}$ and of which the edges are the pairs $(a,f(a))$ with $a\in \mathsf{Dom}(f)$. Indeed, $f$ is TO if and only if the associated directed graph results in a disjoint union of oriented odd cycles and (directed rooted) in-trees (roots are those vertices having outdegree equal to 0), see \cite{graph}.  Instead, $f$ is TL if and only if the associated directed graph results in a disjoint union of loops and (directed rooted) in-trees.

Since every loop is an odd cycle (its lenght is $1$), if $f$ is TL, then $f$ is TO.

An example of a TL function is the partial function $\mathbf{pr}:\mathbb{N}\rightharpoonup \mathbb{N}$ such that $\mathbf{pr}(n)=m$ if and only if $m+1=n$. Indeed, $\delta_{f}(n)=n+1$ for every $n\in \mathbb{N}$. An example of TO total function is the function $\mathsf{inv}_3:\mathbb{N}\rightarrow \mathbb{N}$ defined by $\mathsf{inv}_3(3n)=3n+1$, $\mathsf{inv}_3(3n+1)=3n+2$ and $\mathsf{inv}_3(3n+2)=3n$ for every $n\in \mathbb{N}$.\\

The next proposition is an immediate consequence of the definitions.
\begin{proposition}\label{totalTOTL} Let $f:\mathsf{R}\rightarrow \mathsf{R}$ be a total function. Then,
\begin{enumerate}
\item $f$ is TO if and only if for every $a\in \mathsf{R}$ we have $\mathsf{ord}_{f}(a)=2k+1$ for some $k\in \mathbb{N}$;
\item $f$ is TL if and only if $f=\mathsf{id}_{\mathsf{R}}$.
\end{enumerate}
\end{proposition}

\begin{proposition}\label{comptl}
If $f$ is TO (resp.\ TL), then $f^{n}$ is TO (resp.\ TL) for every $n\in \mathbb{N}$.
\end{proposition}
\begin{proof}
Let $f:\mathsf{R}\rightharpoonup \mathsf{R}$ be a partial function. For $n=0$, $f^n$ is the identity function which is TL.

Let us now consider an element $a\in \mathsf{R}$ and a positive natural number $n$. If $\delta_{f}(a)\in \mathbb{N}^{+}$, then $\delta_{f^n}(a)\in \mathbb{N}^{+}$, since $1\leq \delta_{f^n}(a)\leq \delta_{f}(a)$. Assume now $\mathsf{ord}_f(a)=2k+1$. This means in particular that $f^{(2k+1)n}(a)=a$. From this it follows that there exists $h|2k+1$ (and in particular such $h$ must be odd) such that $\mathsf{ord}_{f^n}(a)=h$. In particular, if $\mathsf{ord}_f(a)=1$, then $\mathsf{ord}_{f^n}(a)=1$. These observations allow to conclude that if $f$ is TO (resp.\ TL), then $f^n$ is TO (resp.\ TL).
\end{proof}

Unfortunately, the collections of TO and TL partial functions are not closed under composition. Consider e.g.\ the functions $g,h:\{0,1,2,3\}\rightharpoonup \{0,1,2,3\}$ defined as follows.
$$\begin{cases}
g(0)=1\\
g(1)\not\downarrow\\
g(2)\not\downarrow\\
g(3)=2\\
\end{cases}
\qquad\qquad
\begin{cases}
h(0)\not\downarrow\\
h(1)=3\\
h(2)=0\\
h(3)\not\downarrow\\
\end{cases}$$
Both $g$ and $h$ are TL, since $\delta_{g}(0)=\delta_g(3)=2$, $\delta_g(1)=\delta_g(2)=1$, $\delta_{h}(0)=\delta_h(3)=1$ and $\delta_h(1)=\delta_h(2)=2$. 
However one has that $h(g(0))=3$ and $h(g(3))=0$ from which it follows that $\mathsf{ord}_{h\circ g}(0)=2$. Thus $h\circ g$ is not TO.

\begin{proposition}\label{inj} Let $\mathsf{R}$ be a non-empty set.
If $f:\mathsf{R}\rightharpoonup \mathsf{R}$ is a TL partial function, then there exists an injective function $j:\{n\in \mathbb{N}|\,n<\Delta(f)\}\rightarrow \mathsf{R}$ if $\Delta(f)\in \mathbb{N}^{+}$ or if $\Delta(f)=+\infty$. 
\end{proposition}
\begin{proof}
We can assume there is at least one element $r\in \mathsf{R}$ with $\delta_f(r)<+\infty$, since if there are no such elements, then $\Delta(f)=1$ by definition and one can just consider a function sending $0$ to some element of $\mathsf{R}$ (which is non-empty by hypothesis).
If $\Delta(f)<+\infty$, then there exists $\overline{a}\in \mathsf{R}$ such that $\delta_f(\overline{a})=\Delta(f)$. In this case we take $j(n):=f^n(\overline{a})$ for every $n<\Delta(f)$. If $\Delta(f)=+\infty$, 
then we take $j(0)$ to be any element of $\mathsf{R}$ and for every $n\in \mathsf{N}$, if we assume to have already defined $j(0),...,j(n)$ we take $j(n+1)$ to be an element of $\mathsf{R}\setminus \{j(k)|\,k=0,...,n\}$ such that $\delta_f(j(n+1))>\delta_f(j(n))$. 
 \end{proof}

\begin{definition} We say that a PAS $\mathcal{R}$ satisfies $\mathbf{TO}$ (resp.\ $\mathbf{TL}$) if the partial function $[r]$ is TO (resp.\ TL) for every $r\in \mathsf{R}$.
\end{definition}

\begin{proposition}\label{Ant->TO}Let $\mathcal{R}$ be a PAS. Then $\mathbf{Ant}\Rightarrow \mathbf{TO}$.
\end{proposition}
\begin{proof}
Let $\mathcal{R}$ be a PAS satisfying $\mathbf{Ant}$ and let $r,a\in \mathsf{R}$. There are a priori three possibilities:
\begin{enumerate}
\item $[r]^n(a)\not\downarrow$ for some $n\in \mathbb{N}^{+}$;
\item there exist natural numbers $n<m$ such that $[r]^{n}(a)=[r]^{m}(a)$;
\item $[r]^{n}(a)\downarrow$ for every $n\in \mathbb{N}$ and $[r]^{(-)}(a):\mathbb{N}\rightarrow \mathsf{R}$ is injective.
\end{enumerate}
Assume we are in case $(3)$. Then $\{[r]^{2n}(a)|\,n\in \mathbb{N}\}\dashv\vdash\{[r]^{2n+1}(a)|\,n\in \mathbb{N}\} $, but  $\{[r]^{2n}(a)|\,n\in \mathbb{N}\}\neq\{[r]^{2n+1}(a)|\,n\in \mathbb{N}\} $, resulting in a contradiction with $\mathbf{Ant}$. Thus $(3)$ does not hold.

Assume now we are in case $(2)$. Let $\overline{n}$ be the minimum natural number such that there exists $k\in \mathbb{N}^{+}$ for which $[r]^{\overline{n}}(a)=[r]^{\overline{n}+k}(a)$, and let $\overline{k}$ be the minimum such $k$. If $\overline{n}>0$, then $\{[r]^{\overline{n}-1}(a),[r]^{\overline{n}}(a),...,[r]^{\overline{n}+\overline{k}-1}(a)\}\dashv\vdash  \{[r]^{\overline{n}}(a),...,[r]^{\overline{n}+\overline{k}-1}(a)\}$, but this two sets are different because of the minimality of $\overline{n}$, and this contradicts $\mathbf{Ant}$. Thus $\overline{n}$ must be $0$. In this case, if $\overline{k}$ is even, then $\{[r]^{2j}(a)|\,j\in \mathbb{N}\}\dashv\vdash\{[r]^{2j+1}(a)|\,j\in \mathbb{N}\}$, but these are distinct sets, and we thus contradict $\mathbf{Ant}$. Thus, $\overline{k}$ must be odd. We can hence conclude that $[r]$ is a TO function.
\end{proof}

Using Propositions \ref{car1w} and \ref{ant} we obtain the following:
\begin{proposition}\label{Ant->1W}Let $\mathcal{R}$ be a PAS. Then $\mathbf{Ant}\Rightarrow \mathbf{1W}$.
\end{proposition}
We can summarize the situation as follows:
$$
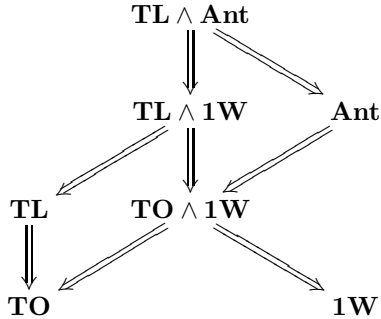
\xymatrix{
		&\mathbf{TL}\wedge \mathbf{Ant}\ar@{=>}[rd]\ar@{=>}[d]		&\\
		&\mathbf{TL}\wedge \mathbf{1W}\ar@{=>}[d]\ar@{=>}[ld]						&\mathbf{Ant}\ar@{=>}[dl]\\
\mathbf{TL}\ar@{=>}[d]		&\mathbf{TO}\wedge \mathbf{1W}\ar@{=>}[ld]\ar@{=>}[rd]				&\\
\mathbf{TO}				&										&\mathbf{1W}\\
}$$

In order to see that this is the best diagram we can draw just consider the following examples:
\begin{enumerate}
\item the following table represents a PAS which satisfies $\mathbf{Ant}$ and $\mathbf{TL}$ (and hence also $\mathbf{TO}$ and $\mathbf{1W}$)
$$\begin{tabular}{|c|c|c|c|}
\hline
$\cdot$	&0	&1	&2	\\
\hline
0		&1	&2	&-	\\
\hline
1		&-	&-	&-	\\
\hline
2		&-	&-	&-	\\
\hline
\end{tabular}$$
\item the following table represents a PAS which satisfies $\mathbf{Ant}$ but not $\mathbf{TL}$ ($\mathbf{TO}$ and $\mathbf{1W}$ are automatically satisfied)
$$\begin{tabular}{|c|c|c|c|}
\hline
$\cdot$	&0	&1	&2	\\
\hline
0		&1	&2	&0	\\
\hline
1		&-	&-	&-	\\
\hline
2		&-	&-	&-	\\
\hline
\end{tabular}$$
\item the following table represents a PAS which satisfies $\mathbf{TL}$ and $\mathbf{1W}$ (and thus automatically $\mathbf{TO}$), but which does not satisfy $\mathbf{Ant}$.
$$\begin{tabular}{|c|c|c|c|c|}
\hline
$\cdot$	&0	&1	&2	&3	\\
\hline
0		&1	&-	&-	&2\\
\hline
1		&-	&3	&0	&-\\
\hline
2		&-	&-	&-	&-\\
\hline
3		&-	&-	&-	&-\\
\hline
\end{tabular}$$
\item the following table represents a PAS which satisfies $\mathbf{TO}$ and $\mathbf{1W}$, but which does not satisfy $\mathbf{Ant}$ nor $\mathbf{TL}$.
$$\begin{tabular}{|c|c|c|c|c|}
\hline
$\cdot$	&0	&1	&2	&3	\\
\hline
0		&1	&-	&-	&2\\
\hline
1		&-	&3	&0	&-\\
\hline
2		&1	&2	&0	&-\\
\hline
3		&-	&-	&-	&-\\
\hline
\end{tabular}$$

\item the following table represents a PAS which satisfies $\mathbf{TO}$, but which does not satisfy $\mathbf{1W}$ (and hence neither $\mathbf{Ant}$) nor $\mathbf{TL}$.
$$\begin{tabular}{|c|c|c|c|}
\hline
$\cdot$	&0	&1	&2	\\
\hline
0		&1	&2	&0	\\
\hline
1		&2	&0	&1	\\
\hline
2		&-	&-	&-	\\
\hline
\end{tabular}$$
\item the following table represents a PAS which satisfies $\mathbf{TL}$ (and thus also $\mathbf{TO}$), but which does not satisfy $\mathbf{1W}$ (and hence neither $\mathbf{Ant}$).
$$\begin{tabular}{|c|c|c|c|}
\hline
$\cdot$	&0	&1	&2	\\
\hline
0		&1	&-	&-	\\
\hline
1		&-	&0	&-	\\
\hline
2		&-	&-	&-	\\
\hline
\end{tabular}$$
\item the following table represents a PAS which satisfies $\mathbf{1W}$, but which does not satisfy $\mathbf{Ant}$ nor $\mathbf{TO}$ (and thus neither $\mathbf{TL}$).
$$\begin{tabular}{|c|c|c|c|c|}
\hline
$\cdot$	&0	&1	&2	&3	\\
\hline
0		&1	&2	&3	&0\\
\hline
1		&-	&-	&-	&-\\
\hline
2		&-	&-	&-	&-\\
\hline
3		&-	&-	&-	&-\\
\hline
\end{tabular}$$
\end{enumerate}
Notice that no one of the examples in the list above satisfies transitivity or totality.
\subsection{Antisymmetry in presence of transitivity}
We assume in this section $\mathcal{R}$ to be transitive. Under this assumption the picture becomes simpler.

\begin{proposition}
Let $\mathcal{R}$ be a transitive PAS. Then $\mathbf{1W}\wedge\mathbf{TO}\Rightarrow \mathbf{TL}$.
\end{proposition}
\begin{proof}
Let us assume $[s]$ to be a TO partial function for every $s$ in a transitive PAS $\mathcal{R}$.  Let $r,a\in \mathsf{R}$ and $n\in \mathbb{N}^{+}$ such that $[r]^{n}(a)=a$. Then we have $\{a\}\vdash \{[r](a)\}\vdash...\vdash \{[r]^{n-1}(a)\}\vdash \{[r]^{n}(a)\}=\{a\}$. From transitivity in $(\mathcal{P}(\mathsf{R}),\vdash)$ we have $\{a\}\vdash \{[r](a)\}$ and $\{[r](a)\}\vdash \{a\}$. Thus, if we assume also $\mathbf{1W}$, we can conclude that $[r](a)=a$. Thus, if $\mathsf{ord}_{[r]}(a)\in \mathbb{N}^{+}$, then $\mathsf{ord}_{[r]}(a)=1$, from which it follows that $[r]$ is a TL partial function.
\end{proof}

\begin{corollary}
Let $\mathcal{R}$ be a transitive PAS. Then $\mathbf{Ant}\Rightarrow \mathbf{TL}$.
\end{corollary}
\begin{proof}
This is a consequence of the previous proposition together with Propositions \ref{Ant->TO} and \ref{Ant->1W}.
\end{proof}
\begin{proposition}
Let $\mathcal{R}$ be a transitive PAS. Then $\mathbf{TL}\wedge\mathbf{1W}\Rightarrow \mathbf{Ant}$.
\end{proposition}
\begin{proof}
Suppose $\mathcal{R}$ is transitive and satisfies $\mathbf{1W}$ and $\mathbf{TL}$. Suppose $A$ and $B$ are subsets of $\mathsf{R}$  and $r,s\in \mathsf{R}$ are such that $r\in A\Rightarrow  B$ and $s\in B\Rightarrow A$. If $A=\emptyset$, then $B=\emptyset$ and $A=B$. Thus, we assume that $A\neq \emptyset$. Let $a$ be an arbitrary element of $A$. We can define a sequence $(a_{n})_{n\in \mathbb{N}}$ of elements of $A$ by taking $a_{0}:=a$ and $a_{n+1}:=s\cdot (r\cdot a_n)$ for every $n\in \mathbb{N}$. Since $\mathcal{R}$ is transitive, there exists $t\in \mathsf{R}$ such that $a_{n+1}=[t](a_{n})$ for every $n\in \mathbb{N}$. Since $[t]$ is TL, then $a_{n+1}=a_{n}$ for every $n\in \mathbb{N}$, that is $a_{n}=a$ for every $n\in \mathbb{N}$. In particular, using $\mathbf{1W}$ and the fact that $s\cdot (r\cdot a)=t\cdot a=a$, we can conclude that $a=r\cdot a$. Since $a$ is an arbitrary element of $A$, we can conclude that $A\subseteq B$. The same argument, interchanging the roles of $A$ and $B$ and the roles of $r$ and $s$, provides the inclusion $B\subseteq A$. Thus $A=B$. \end{proof}

Thus we have the following corollary:
\begin{corollary}
The following are equivalent if $\mathcal{R}$ is a transitive PAS:
\begin{enumerate}
\item $\mathcal{R}$ satisfies $\mathbf{Ant}$;
\item $\mathcal{R}$ satisfies $\mathbf{1W}\wedge \mathbf{TL}$;
\item $\mathcal{R}$ satisfies $\mathbf{1W}\wedge \mathbf{TO}$.
\end{enumerate}
\end{corollary}
Thus under the assumption of transitivity, the situation is simpler, as represented in the following diagram
$$
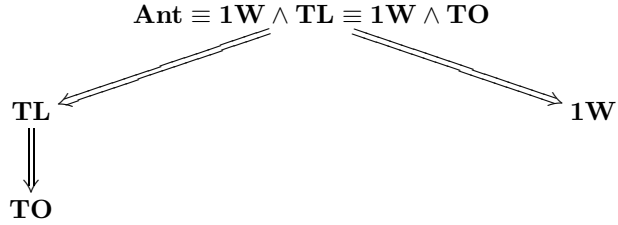
\xymatrix{
				&\mathbf{Ant}\equiv \mathbf{1W}\wedge \mathbf{TL}\equiv \mathbf{1W}\wedge \mathbf{TO} \ar@{=>}[rd]\ar@{=>}[ld]		&\\
\mathbf{TL}\ar@{=>}[d]		&				&\mathbf{1W}\\
\mathbf{TO}						&			&\\
}$$
And this is the best we can do, as proven by the following list of examples:
\begin{enumerate}
\item the following table represents a {\bf transitive} PAS which satisfies $\mathbf{Ant}$ (and hence also $\mathbf{TL}$, $\mathbf{TO}$ and $\mathbf{1W}$)
$$\begin{tabular}{|c|c|c|c|}
\hline
$\cdot$	&0	&1	&2	\\
\hline
0		&1	&2	&-	\\
\hline
1		&2	&-	&-	\\
\hline
2		&-	&-	&-	\\
\hline
\end{tabular}$$

\item the following table represents a {\bf transitive} PAS which satisfies $\mathbf{TL}$ (and thus $\mathbf{TO}$) and which does not satisfy $\mathbf{1W}$ (and hence neither $\mathbf{Ant}$).
$$\begin{tabular}{|c|c|c|c|}
\hline
$\cdot$	&0	&1	&2	\\
\hline
0		&1	&-	&-	\\
\hline
1		&-	&0	&-	\\
\hline
2		&0	&1	&-	\\
\hline
\end{tabular}$$
\item the following table represents a {\bf transitive} PAS which satisfies $\mathbf{TO}$ but which does not satisfy $\mathbf{1W}$ nor $\mathbf{TL}$ (and hence neither $\mathbf{Ant}$).
$$\begin{tabular}{|c|c|c|c|c|}
\hline
$\cdot$	&0	&1	&2	&-\\
\hline
0		&0	&1	&2	&-\\
\hline
1		&1	&2	&0	&-	\\
\hline
2		&2	&0	&1	&-	\\
\hline
3		&-	&-	&-	&3	\\
\hline
\end{tabular}$$
\item the following table represents a {\bf transitive} PAS which satisfies $\mathbf{1W}$, but which does not satisfy $\mathbf{TO}$ (thus not $\mathbf{TL}$ neither $\mathbf{Ant}$).
$$\begin{tabular}{|c|c|c|c|}
\hline
$\cdot$	&0	&1	\\
\hline
0		&1	&1	\\
\hline
1		&-	&-	\\
\hline
\end{tabular}$$
\end{enumerate}
Notice that no one of the examples in this section is total.

\subsection{Antisymmetry in the total case}
\begin{proposition}\label{tltot}Let $\mathcal{R}$ be a total PAS. Then the following are equivalent:
\begin{enumerate}
\item $\mathbf{TL}$;
\item $\cdot=\pi_2:\mathsf{R}\times \mathsf{R}\rightarrow \mathsf{R}$.
\end{enumerate}
In particular, $\mathbf{TL}\Rightarrow \mathbf{Ant}\wedge \mathbf{Comp}(\wedge \mathbf{Trans})\wedge \mathbf{Id}$.
\end{proposition}
\begin{proof}
The fact that 1.\ implies 2.\ is a consequence of Proposition \ref{totalTOTL}(2), while trivially 2.\ implies 1.\ The second part follows from the following observation: if $\cdot=\pi_2$, every element of $\mathsf{R}$ represents the identity on $\mathsf{R}$ (thus $\mathbf{Id}$, $\mathbf{Comp}$ and hence $\mathbf{Trans}$ hold); from this it follows that $\vdash$ coincides with $\subseteq$ and thus $\mathbf{Ant}$ holds as a consequence of Proposition \ref{ant}.
\end{proof}
Thus, we have the following picture:
$$
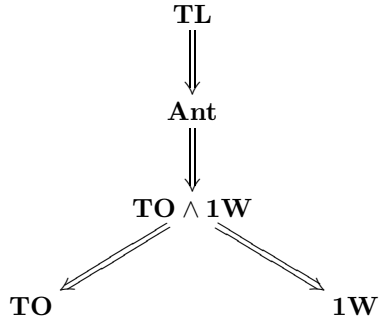
\xymatrix{
		&\mathbf{TL}\ar@{=>}[d]		&\\
		&\mathbf{Ant}\ar@{=>}[d]		&\\
		&\mathbf{TO}\wedge \mathbf{1W}\ar@{=>}[rd]\ar@{=>}[ld]		&\\
\mathbf{TO}			&		&\mathbf{1W}\\		
}$$
This is the best we can do, since:
\begin{enumerate}
\item the following table represents a non-trivial {\bf total} PAS which satisfies $\mathbf{TL}$ (and hence $\mathbf{Ant}$, $\mathbf{TO}$, $\mathbf{TL}$ and $\mathbf{1W}$)
$$\begin{tabular}{|c|c|c|}
\hline
$\cdot$	&0	&1	\\
\hline
0		&0	&1	\\
\hline
1		&0	&1	\\
\hline
\end{tabular}$$
\item  the following table represents a {\bf total} PAS which satisfies $\mathbf{Ant}$ (and hence $\mathbf{TO}$ and $\mathbf{1W}$) but does not satisfy $\mathbf{TL}$.
$$\begin{tabular}{|c|c|c|c|}
\hline
$\cdot$	&0	&1	&2	\\
\hline
0		&1	&2	&0\\
\hline
1		&1	&2	&0\\
\hline
2		&1	&2	&0\\
\hline
\end{tabular}$$

\item  the following table represents a {\bf total} PAS which satisfies $\mathbf{TO}$ and $\mathbf{1W}$, but does not satisfy $\mathbf{Ant}$\footnote{$\{0,1\}\vdash\{2,3\}\vdash\{0,1\}$.} (and thus nor $\mathbf{TL}$).
$$\begin{tabular}{|c|c|c|c|c|c|c|c|c|}
\hline
$\cdot$	&0	&1	&2	&3	&4	&5	&6	&7	\\
\hline
0		&2	&3	&4	&5	&0	&1	&6	&7	\\
\hline
1		&2	&3	&4	&5	&0	&1	&6	&7	\\
\hline
2		&2	&3	&4	&5	&0	&1	&6	&7	\\
\hline
3		&2	&3	&4	&5	&0	&1	&6	&7	\\
\hline
4		&6	&7	&1	&0	&4	&5	&3	&2	\\
\hline
5		&6	&7	&1	&0	&4	&5	&3	&2	\\
\hline
6		&6	&7	&1	&0	&4	&5	&3	&2	\\
\hline
7		&6	&7	&1	&0	&4	&5	&3	&2	\\
\hline
\end{tabular}$$

\item  the following table represents a {\bf total} PAS which satisfies $\mathbf{TO}$ but does not satisfy $\mathbf{1W}$ (and thus nor $\mathbf{Ant}$ neither $\mathbf{TL}$)
$$\begin{tabular}{|c|c|c|c|}
\hline
$\cdot$	&0	&1	&2	\\
\hline
0		&1	&2	&0\\
\hline
1		&1	&2	&0\\
\hline
2		&2	&0	&1\\
\hline
\end{tabular}$$
\item  the following table represents a {\bf total} PAS which satisfies $\mathbf{1W}$ but does not satisfy $\mathbf{TO}$ (and thus nor $\mathbf{TL}$ neither $\mathbf{Ant}$).
$$\begin{tabular}{|c|c|c|c|}
\hline
$\cdot$	&0	&1	&2	\\
\hline
0		&1	&1	&2\\
\hline
1		&0	&2	&2\\
\hline
2		&0	&1	&0\\
\hline
\end{tabular}$$ 
\end{enumerate}
Notice that no one of the previous examples is transitive, except the first (which could not be otherwise, see Proposition \ref{tltot}).

Notice that in case $\mathcal{R}$ is {\bf total and transitive}, the situation collapses to the following:
$$
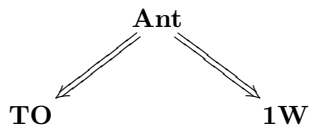
\xymatrix{	&\mathbf{Ant}\ar@{=>}[rd]	\ar@{=>}[ld]	\\
\mathbf{TO}	&	&\mathbf{1W}}$$
Indeed, under these hypotheses, $\mathbf{Ant}\Leftrightarrow \mathbf{TL}$.
Antisymmetry is equivalent to the operation $\cdot$ to be the second projection, while an example of a transitive total PAS satisfying $\mathbf{1W}$ but not satisfying $\mathbf{TO}$ (and thus neither $\mathbf{Ant}$) is described by the following table
$$\begin{tabular}{|c|c|c|}
\hline
$\cdot$	&0	&1	\\
\hline
0		&1	&1	\\
\hline
1		&1	&1	\\
\hline
\end{tabular}$$
and an example of a transitive total PAS satisfying $\mathbf{TO}$ but not satisfying $\mathbf{1W}$ (and thus neither $\mathbf{Ant}$) is described by the following table:
$$\begin{tabular}{|c|c|c|c|}
\hline
$\cdot$	&0	&1	&2	\\
\hline
0		&0	&1	&2	\\
\hline
1		&1	&2	&0\\
\hline
2		&2	&0	&1\\
\hline
\end{tabular}$$
\subsection{Posetal partial applicative structures}
\begin{definition} A PAS $\mathcal{R}$ is \emph{posetal} if $\pi_{\mathsf{R}}:\mathbf{Set}^{op}\rightarrow \mathbf{Bin}$ factorizes through the inclusion functor of the category $\mathbf{Pos}$ of posets into $\mathbf{Bin}$, that is if $\pi_{\mathcal{R}}(I)$ is a poset for every set $I$.
\end{definition}

Glueing together the results in the previous sections we have the following characterization of posetal PASs.
\begin{theorem}\label{carpos} The following are equivalent for a PAS $\mathcal{R}$:
\begin{enumerate}
\item $\mathcal{R}$ is posetal;
\item $\mathcal{R}$ is reflexive and transitive, and $\vdash$ is antisymmetric;
\item $\mathcal{R}$ is reflexive and transitive, and it satisfies $\mathbf{TL}$ and $\mathbf{1W}$.
\end{enumerate}
\end{theorem}

\begin{example}
A class of posetal partial applicative structures can be obtained as follows. Let $f:\mathsf{R}\rightharpoonup \mathsf{R}$ be a TL partial function 
and let $j$ be an injective function from $\{n\in \mathbb{N}|\,n<\Delta(f)\}$ to $\mathsf{R}$ (which always exists if $\Delta(f)\in \mathbb{N}^{+}$ and which exists if $\Delta(f)=+\infty$, 
see Proposition \ref{inj}). 
We define a PAS $\mathcal{R}_{f,j}$ as $(\mathsf{R},\cdot_{f,j})$ where
$$\begin{cases}
j(n)\cdot_{f,j}a\simeq f^{n}(a) \\
b\cdot_{f,j}a\not \downarrow \textrm{ otherwise }\\ 
\end{cases}$$
Such a PAS is reflexive, since $0<\Delta(f)$ and $j(0)\cdot_{f,j}a=f^0(a)=a$ for every $a\in \mathsf{R}$. If $n+m<\Delta(f)$, then $j(n+m)\cdot_{f,j}a\simeq j(n)\cdot_{f,g}(j(m)\cdot a)$ for every $a\in \mathsf{R}$. If $n+m\geq\Delta(f)$, then $j(n)\cdot_{f,g}(j(m)\cdot a)\not\downarrow$ for every $a\in \mathsf{R}$. This is enough to conclude that $\mathcal{R}_{f,j}$ is transitive. Moreover, $\mathbf{1W}$ is clearly satisfied, because $f$ is TL and $\mathbf{TL}$ is satisfied thanks to Proposition \ref{comptl} since $[j(n)]=f^{n}$ for every $n<\Delta(f)$ and $[r]=\emptyset$ otherwise. Thus $\mathcal{R}_{f,j}$ is a posetal PAS, which we will call the posetal PAS \emph{generated} by $f$ through $j$.
\end{example}
\section{Elementary properties in posetal PASs}
\subsection{General properties}
By collecting all the observations about trivial PASs in the previous sections one can conclude that 
\begin{proposition}
Each trivial PAS is a posetal PAS which satisfies all properties we have considered until now.
\end{proposition}
Moreover, total posetal PASs turn out to be very easy to characterize as a consequence of Proposition \ref{tltot} and Theorem \ref{carpos}. 
\begin{proposition}\label{totalPAS}
A total PAS $\mathcal{R}$ is posetal if and only if $\cdot=\pi_2$. 
\end{proposition}

\subsection{Representable endofunctions}
We present here facts concerning properties in Section $2.1.$ in the posetal case.
\begin{proposition}\label{endpos}
Let $\mathcal{R}$ be a posetal PAS. Then
\begin{enumerate}
\item $\mathbf{1W}$ holds;
\item $\mathbf{Id}$ holds;
\item {\bf 1-Total} holds;
\item $\mathbf{TM}\Leftrightarrow \mathbf{Trivial}$;
\item $\mathbf{Const}\Leftrightarrow \mathbf{Trivial}$;
\item $\mathbf{Ext}\wedge \mathbf{Total}\Leftrightarrow \mathbf{Trivial}$;
\item $\mathbf{Total}\Rightarrow \mathbf{Comp}$.
\end{enumerate}
\end{proposition}
\begin{proof}
Items 1.\ and 2.\ follow from Theorem \ref{carpos} and Proposition \ref{ref}. Item 3.\ is a consequence of 2. Items 4.\ and 5.\ follow from Proposition \ref{basic}(2,5) and Theorem \ref{carpos}. Item 6.\ follows by Proposition \ref{totalPAS}, since $(\mathsf{R},\pi_2)$ is extensional if and only if $\mathsf{R}$ has only one element. Item 7.\ is a trivial consequence of Proposition \ref{totalPAS}.

\end{proof}

%
\subsection{Standard algebraic properties}
We present here facts concerning properties in Section $2.1.$ in the posetal case.
\begin{proposition}\label{algpos} Let $\mathcal{R}$ be a posetal PAS, then 
\begin{enumerate}
\item $\mathbf{Ab}\Rightarrow \mathbf{ID}_{R}$;
\item $\mathbf{Total}\wedge \mathbf{ID}_{R}\Leftrightarrow  \mathbf{Trivial}$;
\item $\mathbf{Total}\wedge \mathbf{Ab}\Leftrightarrow \mathbf{Trivial}$;
\item $\mathbf{Total}\Rightarrow \mathbf{Assoc}$;
\item $\mathbf{ID}_{R}\wedge \mathbf{Comp}\Leftrightarrow \mathbf{Trivial}$
\item $\mathbf{Ab}\wedge \mathbf{Comp}\Leftrightarrow \mathbf{Trivial}$
\end{enumerate}
\end{proposition}
\begin{proof}
By Proposition\ \ref{ref}, reflexivity amounts to require the existence of a left identity for $\cdot$; this, combined with $\mathbf{Ab}$, results in the existence of a right identity. Since every posetal PAS is reflexive we can conclude that 1.\ holds. In order to show 2.\ it is sufficient to notice that if $\mathbf{j}$ is a right identity and $\mathcal{R}$ is total posetal, then $a=a\cdot \mathbf{j}=\mathbf{j}$ for every $a\in \mathsf{R}$. Item 3.\ follows from items 1.\ and 2.\
Item 4.\ follows from the fact that if $\mathcal{R}$ is total posetal then $(a\cdot b)\cdot c=c=a\cdot (b\cdot c)$ for every $a,b,c\in \mathsf{R}$.
In order to prove item 5.\ assume $\mathcal{R}$ to safisfy $\mathbf{ID}_R$ and $\mathbf{Comp}$. Assume there exists $a\neq \mathbf{j}$. Since $a\cdot \mathbf{j}=a\neq \mathbf{j}$, then there exists $n$ such that $[a]^{n}(\mathbf{j})\not\downarrow$ since $[a]$ is TL. But we know that by $\mathbf{Comp}$ there exists $r\in \mathsf{R}$ such that $[r]=[a]^n$. For such $r$, we have $r\cdot \mathbf{j}\not\downarrow$, which contradicts the fact that $r\cdot \mathbf{j}=\mathbf{j}$. Thus $a=\mathbf{j}$ for every $a$, from which it follows that $\mathcal{R}$ is trivial. Item 6.\ follows from items 5.\ and 1.\
\end{proof}

%
%

\subsection{Combinators}
\begin{proposition}
Let $\mathcal{R}$ be a posetal PAS. Then:
\begin{enumerate}
\item $\mathbf{K}\Leftrightarrow \mathbf{Trivial}$;
\item $\mathbf{S}\Leftrightarrow \mathbf{Total}$.
\end{enumerate}
\end{proposition}
\begin{proof} Assume $\mathcal{R}$ to be posetal. Since $\mathbf{K}\Rightarrow \mathbf{Const}$, by Proposition \ref{endpos}(5) $\mathbf{K}\Rightarrow \mathbf{Trivial}$.
Let us now assume $\mathcal{R}$ to satisfy $\mathbf{S}$, then $\mathbf{s}\cdot a\downarrow$ for every $a\in \mathsf{R}$, from which, by Proposition \ref{totalTOTL}, it follows that $\mathbf{s}\cdot a=a$ for every $a\in \mathsf{R}$. From this it follows that $a\cdot b\simeq \mathbf{s}\cdot a\cdot b \downarrow$, from which totality follows. Conversely, if $\mathcal{R}$ is total, then by Proposition \ref{totalPAS}, for every $s,a,b,c\in \mathcal{R}$, $s\cdot a\cdot b\cdot c=c=a\cdot c\cdot (b\cdot c)$.
\end{proof}
\begin{corollary}
A PCA $\mathcal{R}$ is posetal if and only if it is trivial.
\end{corollary}

\subsection{The general picture}
Thus, the general picture says that only seven properties are not decided (or not collapsing in triviality) when we assume $\mathcal{R}$ to be a posetal PAS. One is finiteness, indeed we can have finite or infinite posetal PASs. Then we have a group of properties with the relations shown in the diagram which cannot be reversed.
$$
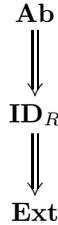
\xymatrix{
\mathbf{Ab}\ar@{=>}[d]		\\
\mathbf{ID}_{R}\ar@{=>}[d]		\\
\mathbf{Ext}				\\
}$$
Indeed:
\begin{enumerate}
\item 
$\begin{tabular}{|c|c|c|c|}
\hline
$\cdot$	&0	&1	&2\\
\hline
0	&0	&1	&2\\
\hline
1	&1	&-	&2\\
\hline
2	&2	&2	&-\\
\hline
\end{tabular}$
represents a non-trivial posetal PAS satisfying $\mathbf{Ab}$;
\item
$\begin{tabular}{|c|c|c|c|}
\hline
$\cdot$	&0	&1	&2\\
\hline
0&0	&1	&2\\
\hline
1&1	&-	&1\\
\hline
2&2	&-	&-\\
\hline
\end{tabular}$
represents a posetal PAS satisfying $\mathbf{ID}_R$ but not $\mathbf{Ab}$;
\item
$\begin{tabular}{|c|c|c|c|}
\hline
$\cdot$	&0	&1	&2\\
\hline
0&0	&1	&2\\
\hline
1&2	&-	&-\\
\hline
2&1	&-	&-\\
\hline
\end{tabular}$
represents a posetal PAS satisfying $\mathbf{Ext}$ but not $\mathbf{ID}_R$;
\item
$\begin{tabular}{|c|c|c|c|}
\hline
$\cdot$	&0	&1	&2\\
\hline
0&0	&1	&2\\
\hline
1&0	&1	&2\\
\hline
2&0	&1	&2\\
\hline
\end{tabular}$
represents a posetal PAS not satisfying $\mathbf{Ext}$.
\end{enumerate}
Moreover, we have also the following group of properties in which converse implications do not hold and no other implication can be added
$$
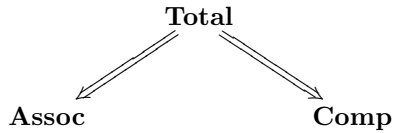
\xymatrix{
		&\mathbf{Total}\ar@{=>}[ld]\ar@{=>}[rd]		&\\
\mathbf{Assoc}		&		&\mathbf{Comp	}	&\\
		&		&	&\\
}$$
since 
\begin{enumerate}
\item the following table represents a non-total posetal PAS satisfying $\mathbf{Comp}$, but not satisfying $\mathbf{Assoc}$ ($(2\cdot 0)\cdot 0\not\downarrow$, while $2\cdot(0\cdot 0)=1$);
$$\begin{tabular}{|c|c|c|c|c|}
\hline
$\cdot$	&0	&1	&2	&3\\
\hline
0&0	&1	&2	&3\\
\hline
1&-	&-	&-	&-\\
\hline
2&1	&-	&-	&-\\
\hline
3&-	&-	&-	&-\\
\hline
\end{tabular}$$

\item the following table represents a non-total posetal PAS satisfying $\mathbf{Assoc}$, but not satisfying $\mathbf{Comp}$;
$$\begin{tabular}{|c|c|c|c|}
\hline
$\cdot$	&0	&1	&2\\
\hline
0&0	&1	&2\\
\hline
1&1	&-	&-\\
\hline
2&2	&-	&-\\
\hline
\end{tabular}$$
\item the first example in the previous list is a posetal PAS which does not satisfy $\mathbf{Assoc}$ ($(2\cdot 1)\cdot 1=2$ but $2\cdot (1\cdot 1)\not\downarrow$) neither $\mathbf{Comp}$.\end{enumerate}

\section{Generated posetal PASs}
We discuss here some of the relevant elementary properties we considered in the previous sections when we restrict to generated posetal partial applicative structures.

Lef us fix a TL partial function $f:\mathsf{R}\rightharpoonup \mathsf{R}$ and an injection $j$ of the set $\mathbb{N}_f:=\{n\in \mathbb{N}|\,n<\Delta(f)\}$ into $\mathsf{R}$ and let us consider the PAS $\mathcal{R}_{f,j}$. Without loss of generality we can assume $\mathsf{R}$ to be of the form $\mathbb{N}_f\cup X$ where $X\cap\mathbb{N}=\emptyset$, and $j$ to be the inclusion of $\mathbb{N}_f$ into $\mathsf{R}$. In this case the partial function $\cdot=\cdot_{f,j}$ of the PAS $\mathcal{R}=\mathcal{R}_{f,j}$ can be described as follows:
$$\begin{cases}
n\cdot a\simeq f^{n}(a)\textrm{ if }n\in\mathbb{N}_f \\
r\cdot a\not \downarrow\textrm{ if }r\in X\\
\end{cases}$$
We have the following results:
\begin{proposition}\label{trivgen}
Let $\mathcal{R}_{f,j}$ be a generated PAS defined as above. Then the following are equivalent
\begin{enumerate}
\item $\mathcal{R}_{f,j}$ is trivial;
\item $\mathcal{R}_{f,j}$ is total;
\item $\Delta(f)=1$ and $X=\emptyset$.
\end{enumerate}
\end{proposition}
\begin{proof}
We already noticed that $\mathbf{Trivial}\Rightarrow  \mathbf{Total}$, that is 1.\ implies 2.\ Let us now assume $\mathcal{R}_{f,j}$ to be total; then $X=\emptyset$ and $i\cdot x\downarrow$ for every $i,x\in \mathbb{N}_f=\mathsf{R}$. If $1\in \mathbb{N}_f$, then this in particular entails that $f(x)=1\cdot x\downarrow$ for every $x\in \mathsf{R}$. Thus $f$ is total and hence must be the identity. But in this case $\Delta(f)=1$ and $1\notin \mathbb{N}_f$ resulting in a contradiction. Thus $\Delta(f)=1$. We have hence shown that 2.\ implies 3.\ Let us now assume $3.$ Then $\mathsf{R}$ contains only $0$ and $0\cdot 0=f^0(0)=0$ by definition. Thus $\mathcal{R}_{f,j}$ is trivial. Hence, 3.\ implies 1.\
\end{proof}

\begin{proposition}\label{idrgen}
Let $\mathcal{R}_{f,j}$ be a generated PAS. Then $\mathcal{R}_{f,j}$ satisfies $\mathbf{ID}_R$ if and only if $X=\emptyset$ and one of the following holds:
\begin{enumerate}
\item $\Delta(f)=1$;
\item $1<\Delta(f)\in \mathbb{N}^{+}$, $f(i)=i+1$ for every $i<\Delta(f)-1$ and $f(\Delta(f)-1)\not\downarrow$.
\end{enumerate}
In particular, if $\mathcal{R}_{f,j}$ satisfies $\mathbf{ID}_R$ then
\begin{enumerate}
\item $\mathcal{R}_{f,j}$ is finite;
\item $\mathcal{R}_{f,j}$ satisfies $\mathbf{Ab}$;
\item $\mathcal{R}_{f,j}$ satisfies $\mathbf{Assoc}$.
\end{enumerate}
\end{proposition}
\begin{proof}
If $\mathcal{R}_{f,j}$ satisfies $\mathbf{ID}_{R}$, then there exists $\mathbf{j}$ such that $r\cdot \mathbf{j}=r$ for every $r\in \mathsf{R}_{f,j}$. Thus in particular $\mathbf{j}=f^{0}(\mathbf{j})=0\cdot \mathbf{j}=0$, from which it follows that $f^{n}(0)=n\cdot 0=n$  for every $n\in \mathbb{N}_f$ and hence $f(m-1)=f(f^{m-1}(0))=f^{m}(0)=m$ for every positive $m \in\mathbb{N}_f$. Moreover, since $r\cdot 0$ must be equal to $r$ for every $r\in \mathsf{R}$, but $r\cdot 0\not\downarrow$ for every $r\in X$, we conclude that $X=\emptyset$. 
Thus we have the following three cases:
\begin{enumerate}
\item $\Delta(f)=1$; 
\item $\Delta(f)>1$ and $f$ is defined as the successor function for all $m<\Delta(f)-1$ and $f(\Delta(f)-1)\not\downarrow$ (since $f$ is TL);
\item $\Delta(f)=+\infty$ and $f(i)=i+1$ for every $i\in \mathbb{N}$. 
\end{enumerate}
However, the last case must be excluded since the successor function from $\mathbb{N}$ to $\mathbb{N}$ is not TL (the only total TL functions are identities by Proposition \ref{totalTOTL}).
In the first case, we get the trivial PAS, while in the second case we have $i\cdot j=i+j$ if $i+j<\Delta(f)$, while $i\cdot j\not\downarrow$, otherwise. In particular, $\mathcal{R}_{f,j}$ is always finite, abelian and associative.
\end{proof}

\begin{proposition}\label{extgen}
Let $\mathcal{R}_{f,j}$ be a generated PAS. Then the following are equivalent:
\begin{enumerate}
\item $\mathcal{R}_{f,j}$ satisfies $\mathbf{Ext}$;
\item $X$ has at most one element.
\end{enumerate}
\end{proposition}
\begin{proof}
Clearly 1.\ implies 2.\ since all the elements of $X$ represent the empty function. Conversely, if $X$ has at most one element, this element represents the empty function, while every $n\in \mathbb{N}_f$ represents $f^n$. Since $f$ is TL, then $f^{i}\neq f^{j}$ for every $0\leq i<j<\Delta(f)$ and $f^{i}$ is never empty for $i<\Delta(f)$. Thus $\mathcal{R}_{f,j}$ satisfies $\mathbf{Ext}$.
\end{proof}
\begin{proposition}\label{compgen}
Let $\mathcal{R}_{f,j}$ be a generated PAS. Then $\mathcal{R}_{f,j}$ satisfies $\mathbf{Comp}$ if and only if one of the following holds:
\begin{enumerate}
\item $\Delta(f)=1$;
\item $X\neq \emptyset$, $1<\Delta(f)\in \mathbb{N}^{+}$ and $\delta_f(a)\in \mathbb{N}^{+}$ for every $a\in \mathsf{R}$;
\item $\Delta(f)=+\infty$.
\end{enumerate}
\end{proposition}
\begin{proof}
Let  $\mathcal{R}_{f,j}$ be a generated PAS and let us consider the following three cases: $\Delta(f)=1$, $1<\Delta(f)\in \mathbb{N}^{+}$ and $\Delta(f)=+\infty$. In the first case $\mathbf{Comp}$ holds since the representable partial functions are only the identity and (possibly) the empty function. In the third case $\mathbf{Comp}$ holds since the representable partial functions are the iterated compositions $f^{n}$ for $n\in \mathbb{N}$ and (possibly) the empty function. In the second case we distinguish two subcases: 
\begin{enumerate}
\item $\delta_f(x)\in \mathbb{N}^{+}$ for every $x\in \mathsf{R}$. In this case in order to get $\mathbf{Comp}$ we need the empty function to be representable, and this amounts to require $X\neq \emptyset$;
\item $\delta_f(x)=+\infty$ for some $x\in \mathsf{R}$. For such $x$, we get $f^n(x)=x$ for every $n\in \mathbb{N}$. In particular, $f^{\Delta(f)}(x)=x$, while $f^{\Delta(f)}(y)\not \downarrow$ for every $y$ such that $\delta_{f}(y)\in \mathbb{N}^{+}$. Thus $f^{\Delta(f)}$ is not represented by any element of $\mathsf{R}$, while $f$ is represented by $1$. Hence this case must be excluded.
\end{enumerate}
\end{proof}
Putting together the information in the previous propositions and in the previous section we obtain the following picture.
$$
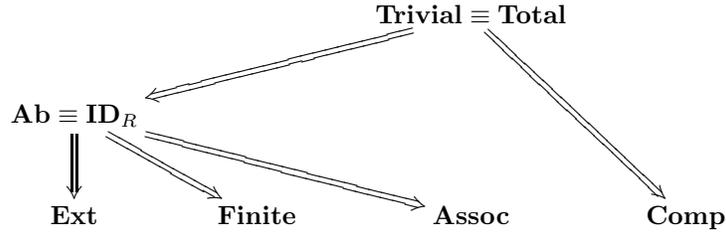
\xymatrix{
							&			&\mathbf{Trivial}\equiv \mathbf{Total}\ar@{=>}[rdd]\ar@{=>}[lld]			&\\
\mathbf{Ab}\equiv \mathbf{ID}_{R}\ar@{=>}[drr]\ar@{=>}[dr]\ar@{=>}[d]		&		&	&\\
\mathbf{Ext}			&\mathbf{Finite}			&\mathbf{Assoc}		&\mathbf{Comp}\\
}$$

If we limit ourselves to the finite case, we obtain the following results:
\begin{proposition}
Let $\mathcal{R}=\mathcal{R}_{f,j}$ be a finite generated posetal PAS. Then:
\begin{enumerate}
\item $\mathcal{R}$ satisfies $\mathbf{Ext}$ if and only if $\Delta(f)=1$ and $X$ has at most one element, or $\Delta(f)=n>1$ and (the directed graph associated to) $f$ has one of the following forms:
$$
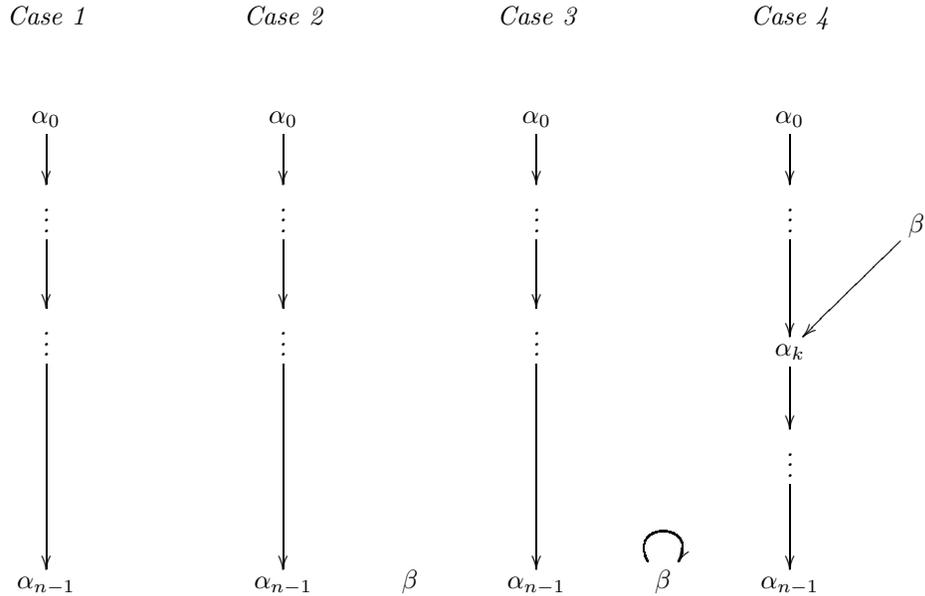
\xymatrix{
\textrm{Case 1}		&		&\textrm{Case 2}		&		&\textrm{Case 3}	&		&\textrm{Case 4}	&\\
\alpha_{0}\ar[d]		&		&\alpha_{0}\ar[d]		&		&\alpha_{0}\ar[d]	&		&\alpha_{0}\ar[d]	&	\\
\vdots	\ar[d]		&		&\vdots\ar[d]			&		&\vdots	\ar[d]		&		&\vdots\ar[d]		&\beta\ar[ld]\\
\vdots	\ar[dd]		&		&\vdots\ar[dd]			&		&\vdots	\ar[dd]		&		&\alpha_k\ar[d]		&\\
&				&		&					&						&		&\vdots\ar[d]			&\\
\alpha_{n-1}		&		&\alpha_{n-1}			&\beta	&\alpha_{n-1}		&\beta\ar@(ul,ur)	&\alpha_{n-1}		&\\
}$$
with $k\in \{1,...,n-1\}$.

\item $\mathcal{R}$ satisfies $\mathbf{Comp}$ if and only if $\Delta(f)=1$ or $1<\Delta(f)=n\in \mathbb{N}$, $X\neq \emptyset$ and $\delta_f(a)\in \mathbb{N}^{+}$ for every $a\in \mathsf{R}$. 
\item $\mathcal{R}$ satisfies both $\mathbf{Ext}$ and $\mathbf{Comp}$ if and only if $\Delta(f)=1$ and $X$ has at most one element, or $1<\Delta(f)=n\in \mathbb{N}$ and we are in Case 2 or Case 4 above.
\item $\mathcal{R}$ safisfies $\mathbf{Ext}$ and $\mathbf{Assoc}$ if and only if $\Delta(f)=1$ and $X$ has at most one element, or $1<\Delta(f)=n\in \mathbb{N}$ and we are in Case 1 or Case 2 above, with
$\alpha_k=k$ for every $k<n$.

\item $\mathcal{R}$ safisfies $\mathbf{Ext}$, $\mathbf{Assoc}$ and $\mathbf{Comp}$ if and only if $\Delta(f)=1$ and $X$ has at most one element, or $\Delta(f)=n>1$ and we are in Case 2 above with $\alpha_k=k$ for every $k<n$.
\end{enumerate}
\end{proposition}
\begin{proof}
1.\ and 2.\ immediately follow from Propositions \ref{extgen} and \ref{compgen} by adding the finiteness constraint. 3.\ immediately follows from 1.\ and 2.\
Item 5.\ is an immediate consequence of items 3.\ and 4.\ 
So it remains to prove item 4.\ 

 First, it is immediate to verify that if $\Delta(f)=1$ and $X$ has at most one element, then $\mathsf{R}$ is associative. It hence remain to consider Cases 1,2,3,4 assuming $\mathbf{Assoc}$ to hold. In Cases 2,3 and 4 the only element $x$ of $X$ cannot be in $\{\alpha_0,...,\alpha_{n-2}\}$ and in Case $4$ it cannot be $\beta$, since otherwise we would have $1\cdot (x\cdot \alpha_0)\not\downarrow$, but $(1\cdot x)\cdot \alpha_0\downarrow$. Thus $x$ can be $\alpha_{n-1}$ or $\beta$ in Cases 2 and 3, or $\alpha_{n-1}$ in Case 4. We then observe that, for every $k\in \{0,...,n-2\}$, we have $\alpha_k\cdot \alpha_0=\alpha_k\cdot (0\cdot \alpha_0)\simeq(\alpha_k\cdot 0) \cdot \alpha_0$, and since $\alpha_k\cdot \alpha_0\downarrow$, it follows that $\alpha_{k}\cdot 0=\alpha_{k}$, from which it follows that $\alpha_0=0$ and thus $\alpha_k=k$ for every $k\in \{0,...,n-2\}$, since $\alpha_{\alpha_k}=\alpha_{k}\cdot \alpha_0=\alpha_{k}\cdot 0=\alpha_k$ for such a $k$. Let us now assume that $x=\alpha_{n-1}$ in one of Cases 2,3,4. Then, $((n-1)\cdot 0) \cdot 0\not\downarrow$, but $(n-1)\cdot (0\cdot 0)\downarrow$. This in particular let us exclude Case 4 in toto. However, also Case 3 can be excluded, since if $x=\beta$, then for $k,j\in \mathbb{N}_f$ with $k+j>n-1$ we have $k\cdot (j\cdot x)\downarrow$, but $(k\cdot j)\cdot x\not\downarrow$. So let us keep only Case 2 with $\alpha_k=k$ for every $k=0,...,n-1$. For Case 1, we can reason as above and show that if $\mathcal{R}$ is associative, then $\alpha_k=k$ for every $k=0,...,n-1$. To conclude we must show that each of these two cases is in fact associative. This is a simple verification for Case 1 and for Case 2, once one restrict to values in $\{0,...,n-1\}$. One can also easily check that $a\cdot (b\cdot c)\not\downarrow$ and $(a\cdot b)\cdot c\not\downarrow$ in Case 2 if at least one among $a$, $b$ and $c$ is $x$.

\end{proof}
\section{On completeness for preorderal PASs}
In this subsection we will always assume $\mathcal{R}$ to be a preorderal PAS. 
\begin{definition} Let $\leq$ be a preorder relation on a set $A$ and $(a_i)_{i\in I}$ a set-indexed
family of elements of $A$. An element $b\in A$ is a \emph{supremum} for $(a_i)_{i\in I}$ if $a_i\leq b$  for every $i\in I$, and for every $c\in A$, if $a_i\leq c$ for every $i\in I$, then $b\leq c$. The preorder $\leq$ on $A$ is complete if every set-indexed family of elements of $A$ has a supremum.
\end{definition}
If $(A,\leq)$ is complete, it is well known that each family $(a_i)_{i\in I}$ of its elements has also an \emph{infimum}, that is an element $d\in A$ such that:
\begin{enumerate}
\item $d\leq a_i$ for every $i\in I$;
\item for every $c\in A$, if $c\leq a_i$ for every $i\in I$, then $c\leq d$.
\end{enumerate}
Indeed, one can just take $d$ to be a supremum for $\{c\in A|\;\forall i\in I\,(c\leq a_i)\}$. Similarly, one can show that each preorder having infima for all set-indexed families of its elements is complete.
We will use $\bigvee_{i\in I} a_i$ and $\bigwedge_{i\in I} a_i$ to denote a supremum and an infimum for $(a_i)_{i\in I}$, respectively. 

A PAS $\mathcal{R}$ is \emph{complete} if $\pi_\mathcal{R}$ factors through the subcategory $\mathbf{cPord}$ of $\mathbf{Bin}$ of  complete preorders, more precisely if $\pi_{\mathcal{R}}(I)$ is complete for every $I$ and the reindexing maps $\pi_{\mathcal{R}}(f)$ preserve suprema and infima.

In order to provide a necessary condition for a PAS to be complete we need the following lemma.

\begin{lemma}\label{l1}Let $(Y_r)_{r\in \mathsf{R}}$ be a family of non-empty subsets of $\mathsf{R}$ such that 
$$\left|\prod_{r\in \mathsf{R}}Y_r\right|>|\mathsf{R}|$$
and let $(X_r)_{r\in \mathsf{R}}$ be a family of pairwise disjoint non-empty subsets of $\mathsf{R}$. Then, there exists a function $$\varphi : \bigcup_{r\in \mathsf{R}}X_{r}\rightarrow \mathsf{R}$$ such that 
\begin{enumerate}
\item for every $r\in \mathsf{R}$ and every $x\in X_r$, $\varphi(x)\in Y_r$;
\item for every $r\in \mathsf{R}$ and every $x,y\in X_r$, $\varphi(x)=\varphi(y)$;
\item there is no $s\in \mathsf{R}$ such that $[s]\supseteq \varphi$.
\end{enumerate}
\end{lemma}
\begin{proof} The result follows from the fact that $|\{f : \bigcup_{r\in \mathsf{R}}X_{r}\rightarrow \mathsf{R}|\,\textrm{(1) and (2) holds}\}|=\left|\prod_{r\in \mathsf{R}}Y_r\right|>|\mathsf{R}|$. \end{proof}

Moreover, it will be useful to introduce the following notation.
\begin{definition} For every $r\in \mathsf{R}$ we define $\mathsf{R}(r)$ to be $\{s\cdot r|\,s\in \mathsf{R}\}$.
\end{definition}
\noindent Notice that $r\in \mathsf{R}(r)$ for every $r\in \mathsf{R}$, since $\mathcal{R}$ is reflexive. 

We are now ready to state and prove the main result of this section:

\begin{theorem}\label{comp1} If $\left|\prod_{r\in \mathsf{R}}\mathsf{R}(r)\right|>|\mathsf{R}|$, 
then $\pi_{\mathcal{R}}(\mathsf{R})=(\mathcal{P}(\mathsf{R})^{\mathsf{R}},\vdash_{\mathsf{R}})$ is not complete. In particular, if $\mathcal{R}$ is complete, then $\left|\prod_{r\in \mathsf{R}}\mathsf{R}(r)\right|\leq |\mathsf{R}|$.
\end{theorem}
\begin{proof} 
Let $\mathcal{R}$ satisfy $\left|\prod_{r\in \mathsf{R}}\mathsf{R}(r)\right|>|\mathsf{R}|$ and for every $a\in \mathsf{R}$ let us consider the  
function $\varphi_a:\mathsf{R}\rightarrow \mathcal{P}(\mathsf{R})$ defined as follows: 
$$\varphi_{a}(x)=\begin{cases}\{a\}\textrm{ if }x=a\\ \emptyset \textrm{ if }x\neq a\end{cases}$$
Suppose $\psi$ to be a supremum for the family $(\varphi_{a})_{a\in\mathsf{R}}$.

If we define the function $\mathsf{sgl}$ as in the proof of Proposition \ref{ref} then $\mathbf{i}\Vdash (\varphi_a \vdash_{\mathsf{R}} \mathsf{sgl})$ for every $a\in \mathsf{R}$, where $\mathbf{i}$ is an element of $\mathcal{R}$ such that $\mathbf{i}\cdot x=x$ for every $x\in \mathsf{R}$, which exists since $\mathcal{R}$ is reflexive. Thus $\psi\vdash_{\mathsf{R}} \mathsf{sgl}$. In particular, if $\psi(a)\cap \psi(b)\neq \emptyset$, then $a=b$. Moreover, since $\varphi_a \vdash \psi$ for every $a\in \mathsf{R}$, we have that $\psi(a)\neq \emptyset$ for every $a\in \mathsf{R}$.

We are in the conditions for applying Lemma \ref{l1} with respect to the families $(\mathsf{R}(a))_{a\in \mathsf{R}}$ and $(\psi(a))_{a\in \mathsf{R}}$. Thus there exists a function $\varphi: \bigcup_{a\in \mathsf{R}} \psi(a)\rightarrow \mathsf{R}$ such that $\varphi(x)\in \mathsf{R}(a)$ for every $a\in \mathsf{R}$ and $x\in \psi(a)$,  $\varphi(x) = \varphi(y)$ for every
$a\in \mathsf{R}$ and every $x,y\in\psi(a)$, and for which there is no $s\in\mathsf{R}$ such that $\varphi\subseteq [s]$. 

Let $\widetilde{\varphi} : \mathsf{R}\rightarrow \mathcal{P}(\mathsf{R})$ be the function defined by
$\widetilde{\varphi}(a):=\{\varphi(x)|\,x\in\psi(a)\}$ for every $a\in \mathsf{R}$.

Since for every $a\in A$ we have that $\varphi(x)\in \mathsf{R}(a)$ for every $x\in \psi(a)$ and that $\widetilde{\varphi}(a)$ is a singleton, we conclude that $\varphi_{a}\vdash_{\mathsf{R}}\widetilde{\varphi}$ for every $a \in  \mathsf{R}$. From this it follows that $\psi \vdash_{\mathsf{R}}\widetilde{\varphi}$, that is there exists $s\in \mathsf{R}$ such that for every $a\in \mathsf{R}$ and for every $x\in \psi(a)$, $s\cdot x=\varphi(x)$, that is $\varphi\subseteq [s]$. This is a contradiction. Thus, the family $\{\varphi_{a}\}_{a\in \mathsf{R}}$ has no supremum in $\pi_{\mathcal{R}}(\mathsf{R})$.\end{proof}

%
%
%
%
\begin{corollary} If $\mathcal{R}$ has at least two elements and it is complete, then it does not satisfy $\mathbf{TM}$. In particular, every non-trivial PCA is not complete. 
\end{corollary}
\begin{proof}
If $\mathcal{R}$ is a preorderal totally matching PAS, then $\mathsf{R}(r)=\mathsf{R}$ for every $r\in \mathsf{R}$, thus if $\mathsf{R}$ has at least two elements, then $\left|\prod_{r\in \mathsf{R} }\mathsf{R}(r)\right|=\left|\mathsf{R}^{\mathsf{R}}\right|>|\mathsf{R}|$. Hence, as a consequence of Theorem \ref{comp1}, $\mathcal{R}$ is not complete.
\end{proof}

We now introduce the following notation.
\begin{definition}
Let $\mathcal{R}$ be a PAS. We define $\mathsf{R}_2$ as the set of all $r\in \mathsf{R}$ such that $\mathsf{R}(r)$ has at least two elements.
\end{definition}
The following proposition characterize those preorderal PASs $\mathcal{R}$ for which $\mathsf{R}_2=\emptyset$.

\begin{proposition}
Let $\mathcal{R}$ be a preorderal PAS. The following are equivalent
\begin{enumerate}
\item $\mathsf{R}_2=\emptyset$;
   \item for every $r,a\in \mathsf{R}$, if $r\cdot a\downarrow$, then $r\cdot a=a$;
    \item for every $a,b\in \mathsf{R}$, if $a\neq b$, then $\emptyset$ is an infimum for $\{a\}$ and $\{b\}$ in $(\mathcal{P}(\mathsf{R}),\vdash)$;
    \item for every $A,B\subseteq \mathsf{R}$, $A\cap B$ is an infimum for $A$ and $B$ in $(\mathcal{P}(\mathsf{R}),\vdash)$;
    \item $\vdash$ coincides with $\subseteq$;
    \item $\pi_{\mathcal{R}}$ is the Boolean-valued tripos on $\mathbf{Set}$ relative to the complete Boolean algebra $(\mathcal{P}(\mathsf{R}),\subseteq)$.
\end{enumerate}
\end{proposition}
\begin{proof}
1.\ is equivalent to 2.\ , since for a preorderal PAS $\mathsf{R}_2=\emptyset$ if and only if $\mathsf{R}(a)=\{a\}$ for every $a\in \mathsf{R}$. From 2.\ item 5.\ follows immediately since $\mathcal{R}$ is preorderal. Moreover 4.\ follows from 5.\ since intersections are exactly binary infima with respect to the inclusion relation in $\mathcal{P}(\mathsf{R})$. Item 3.\ is a special case of 4. Finally, we prove that  3.\ implies 2.\ by contradiction. So let us assume that there exist $r,a,b\in \mathsf{R}$ such that $a\neq b$ and $r\cdot a=b$. From this it follows that $\{a\}\vdash \{b\}$. Since $\mathcal{R}$ is preorderal we also know that $\{a\}\vdash \{a\}$. By $(3)$ we have that $\emptyset$ is an infimum for $\{a\}$ and $\{b\}$, from which it follows that $\{a\}\vdash \emptyset$ which leads to a contradiction. Finally, obviously from 2.\ one can derive 6.\ and from 6.\ it follows 5.\ as it is clear by working on a fiber over a singleton. 
\end{proof}
From the previous proposition it follows that if $\mathsf{R}_2=\emptyset$, then $\mathcal{R}$ is complete. Indeed, the fibers of the Boolean-valued tripos on $\mathbf{Set}$ relative to a complete Boolean algebra are complete Boolean algebras and its reindexing maps preserve infima and suprema.

We present now another corollary of Theorem \ref{comp1} providing a strict upper bound for the cardinality of $\mathsf{R}_2$ in the case in which $\mathcal{R}$ is complete.

\begin{corollary}
 If $\mathcal{R}$ is complete, then $|\mathsf{R}_2|<|\mathsf{R}|$. In particular, if $\mathcal{R}$ is complete and $\mathsf{R}$ is countably infinite, then there is at most a finite number of elements of $\mathsf{R}$ such that $|\mathsf{R}(r)|\geq 2$.
\end{corollary}
\begin{proof}
Let $\mathcal{R}$ be a preorderal PAS. 
$$\left|\prod_{r\in \mathsf{R}}\mathsf{R}(r)\right|\geq \left|\prod_{r\in \mathsf{R}_2}\mathsf{R}(r)\right|\geq |\{0,1\}^{{\mathsf{R}}_2}|.$$
Thus, if $|\mathsf{R}_2|=|\mathsf{R}|$, we have that $\left|\prod_{r\in \mathsf{R}}\mathsf{R}(r)\right|=  |\{0,1\}^\mathsf{R}|>|\mathsf{R}|$ and hence that $\mathcal{R}$ is not complete by Theorem \ref{comp1}.
\end{proof}
Here is an example of use of the previous corollary.
\begin{example}
Let us consider the PAS $(\mathbb{N},\cdot)$ where $\cdot$ is the usual product of natural numbers. 
This PAS is preorderal since it is a monoid. It is not totally matching since e.g.\ two distinct prime numbers cannot be obtained each other via a multiplicative function, but $\mathsf{R}(n)=\{kn|\,k\in \mathbb{N}\}$ has infinitely many elements for every positive natural number $n$. Thus, $(\mathbb{N},\cdot)$ is not complete. 
\end{example}
We present here other three consequences of Theorem \ref{comp1}. 
\begin{remark} If $|\mathsf{R}|=2$ and $\mathcal{R}$ is complete, then $\mathsf{R}_2=\emptyset$. Indeed, let us assume that $\mathcal{R}$ is a preorderal PAS, $\mathsf{R}=\{0,1\}$ and $\mathsf{R}_2\neq \emptyset$. If $|\mathsf{R}_2|=2$, then $\mathcal{R}$ is not complete since $\prod_{r\in \mathsf{R}}\mathsf{R}(r)=4>2$. If  $|\mathsf{R}_2|=1$, then we can consider without loss of generality only the following two cases:
$$\begin{tabular}{|c|c|c|}
\hline
$\cdot$		&$0$		&$1$\\
\hline
$0$		&$0$		&$1$\\
\hline
$1$			&$1$			&$1$\\
\hline
\end{tabular}
\qquad
\begin{tabular}{|c|c|c|}
\hline
$\cdot$		&$0$		&$1$\\
\hline
$0$		&$0$		&$1$\\
\hline
$1$			&$1$			&-\\
\hline
\end{tabular}$$
In both cases, $\pi_{\mathcal{R}}(\{0,1\})$ is not complete, since $(\{0\},\{1\})$ and $(\{1\},\{0\})$ have no infimum.
\end{remark}
\begin{corollary}
If $\mathcal{R}$ is a complete PAS with at least two elements such that $|\mathsf{R}_2|>0$, then $|\mathsf{R}_2|>1$.
\end{corollary}
\begin{proof}
Let $\mathcal{R}$ be a complete PAS with at least two elements in which there is exactly one element $\alpha\in \mathsf{R}$ with $|\mathsf{R}(\alpha)|\geq 2$. By hypothesis there exists $r\in \mathsf{R}\setminus \{\alpha\}$ such that $r\in \mathsf{R}(\alpha)$. 
There are two possibilities:
\begin{enumerate}
\item there exists $s\in \mathsf{R}$ such that $s\cdot r=s\cdot \alpha=r$;
\item there is no such an $s$.
\end{enumerate} 
In any of these two cases we consider a PAS $\mathcal{R}'=(\{\alpha,r\},\bullet)$
where $\bullet$ is defined by the following tables in the first and in the second case, respectively.
$$\begin{tabular}{|c|c|c|}
\hline
$\bullet$		&$\alpha$		&r\\
\hline
$\alpha$		&$\alpha$		&r\\
\hline
r			&r			&r\\
\hline
\end{tabular}
\qquad
\begin{tabular}{|c|c|c|}
\hline
$\bullet$		&$\alpha$		&r\\
\hline
$\alpha$		&$\alpha$		&r\\
\hline
r			&r			&-\\
\hline
\end{tabular}$$
For every $I$, $\pi_{\mathcal{R}'}(I)\subseteq \pi_{\mathcal{R}}(I)$ and
\begin{enumerate}
\item for every $\varphi,\psi\in \pi_{\mathcal{R}'}(I)$, $\psi\vdash_{I}^{\mathcal{R}'} \varphi$ if and only if $\psi\vdash_{I}^{\mathcal{R}} \varphi$.
\item for every $\varphi\in \pi_{\mathcal{R}'}(I)$ and $\psi\in \pi_{\mathcal{R}}(I)$, if $\psi\vdash_{I}^{\mathcal{R}} \varphi$, then $\psi\in  \pi_{\mathcal{R}'}(I)$ and $\psi\vdash_{I}^{\mathcal{R}'} \varphi$.
\end{enumerate}
From this it follows that all infima of families of elements of $\pi_{\mathcal{R}'}(I)$ in $\pi_{\mathcal{R}}(I)$, are infima for the same families in $\pi_{\mathcal{R}'}(I)$. 
In particular, if $\mathcal{R}$ is complete, then $\pi_{\mathcal{R}'}(I)$ must be complete for every set $I$. 
However $\pi_{\mathcal{R}'}(\{\alpha,r\})$ is not complete as a consequence of the previous remark.
\end{proof}

\begin{remark}
If $\mathcal{R}$ is complete and $|\mathsf{R}|=3$ then $\mathsf{R}_2=\emptyset$.
Indeed, if $|\mathsf{R}_2|>0$ and $\mathcal{R}$ is complete, then $|\mathsf{R}_2|\geq 2$ by the previous corollary. But in this case $\left|\prod_{r\in \mathsf{R}}\mathsf{R}(r)\right|\geq 4>3$. Thus by Theorem \ref{comp1} we have a contradiction.
\end{remark}

After this discussion it is clear there is a natural open problem.
\begin{open}\label{open} Is there any complete partial applicative structure $\mathcal{R}$ for which $\pi_{\mathcal{R}}$ is not the Boolean-valued tripos relative to $(\mathcal{P}(\mathsf{R}),\subseteq)$?
\end{open}
In this section we have  restricted the search area for possible examples, however a definitive answer is not currently available.

\section{Conclusions}
In this paper we investigate two particular kinds of partial applicative structures: those giving rise to indexed posets and those giving rise to complete indexed preorders. 
In the first case, we were able to describe exactly the antisymmetry phenomenon in terms of basic computational properties of the partial applicative structure, at least in the case in which transitivity holds. This lead to a characterization of posetal PASs which is user friendly and allow to check quite easily when a preorderal PAS is antisymmetric. The case of complete PASs is instead still under a veil of fog. In the last section we provided some necessary conditions essentially based on cardinality arguments which significantly restricted the search area for possible examples of complete PASs. Clearly a future goal for such a research direction should be to give a definitive answer to Open Problem \ref{open}.

\subsection*{Acknowledgments} The author would like to thank Francesco Ciraulo and Maria Emilia Maietti for fruitful discussions on the topic of the present paper.

\bibliographystyle{alpha}
\bibliography{biblioPAS}
\end{document}